\documentclass[reqno]{article}
\usepackage{amsmath}
\usepackage{amsthm,amscd,amssymb,amsfonts, amsbsy}
\usepackage{latexsym}
\usepackage{mathrsfs}
\usepackage{mathabx}
\usepackage{esint}
\usepackage{color}
\usepackage{enumerate}

\numberwithin{equation}{section}

\theoremstyle{plain}
\newtheorem{theorem}[equation]{Theorem}
\newtheorem{lemma}[equation]{Lemma}

\theoremstyle{definition}

\newtheorem*{acknowledgment}{Acknowledgment}

\theoremstyle{remark}
\newtheorem{remark}[equation]{Remark}

\newcommand{\VMO}{\mathrm{VMO}}
\newcommand{\BMO}{\mathrm{BMO}}

\newcommand{\supp}{\operatorname{supp}}
\newcommand{\dist}{\operatorname{dist}}
\newcommand{\diam}{\operatorname{diam}}

\newcommand{\esssup}{\operatorname*{ess\,sup}}

\newcommand\bR{\mathbb{R}}

\newcommand\cQ{\mathcal{Q}}

\newcommand\sB{\mathscr{B}}
\newcommand\sC{\mathscr{C}}

\newcommand\sL{\mathscr{L}}

\newcommand\sP{\mathscr{P}}

\newcommand\sV{\mathscr{V}}
\newcommand\sW{\mathscr{W}}

\providecommand{\ip}[1]{\left\langle#1\right\rangle}
\providecommand{\set}[1]{\{#1\}}
\providecommand{\Set}[1]{\left\{#1\right\}}
\providecommand{\bigset}[1]{\bigl\{#1\bigr\}}

\providecommand{\abs}[1]{\lvert#1\rvert}
\providecommand{\Abs}[1]{\left\lvert#1\right\rvert}
\providecommand{\bigabs}[1]{\bigl\lvert#1\bigr\rvert}

\providecommand{\norm}[1]{\lVert#1\rVert}
\providecommand{\Norm}[1]{\left\lVert#1\right\rVert}

\providecommand{\tri}[1]{\vvvert#1\vvvert}

\renewcommand{\vec}[1]{\boldsymbol{#1}}
\renewcommand{\qedsymbol}{$\blacksquare$}

\begin{document}
\title{Green's function for second order parabolic systems with Neumann boundary condition\thanks{\footnotesize 
This work was supported NRF grants No. 2010-0008224 and R31-10049 (WCU program).}}

\author{Jongkeun Choi\thanks{\footnotesize Department of Mathematics, Yonsei
University, Seoul 120-749,  Republic of Korea (cjg@yonsei.ac.kr).} \and
Seick Kim\footnote{Department of Computational Science and Engineering, Yonsei University, Seoul 120-749,  Republic of Korea (kimseick@yonsei.ac.kr).}}
\date{}
\maketitle

\begin{abstract}
We study the Neumann Green's function for second order parabolic systems in divergence form with time-dependent measurable coefficients in a cylindrical domain $\mathcal{Q}=\Omega\times (-\infty,\infty)$, where $\Omega\subset \bR^n$ is an open connected set such that a multiplicative Sobolev embedding inequality holds there.
Such a domain includes, for example, a bounded Sobolev extension domain, a special Lipschitz domain, and an unbounded domain with compact Lipschitz boundary. 
We construct the Neumann Green's function in $\mathcal{Q}$ under the assumption that weak solutions of the systems satisfy an interior H\"older continuity estimate. 
We also establish global Gaussian bounds for Neumann Green's function under an additional assumption that weak solutions with zero Neumann data  satisfy a local boundedness estimate.
In the scalar case, such a local boundedness estimate is a consequence of De Giorgi-Moser-Nash theory holds for equations with bounded measurable coefficients in Sobolev extension domains, while in the vectorial case, one may need to impose further regularity assumptions on the coefficients of the system as well as on the domain to obtain such an estimate.
We present a unified approach valid for both the scalar and vectorial cases and discuss some applications of our results including the construction of Neumann functions for second order elliptic systems with measurable coefficients in two dimensional domains.
\end{abstract}

\noindent {\footnotesize {\bf AMS subject classifications.} 35K08, 35K10, 35K40}

\noindent {\footnotesize {\bf Key words.} Green function; Neumann function; Neumann heat kernel; Neumann problem; parabolic system; measurable coefficients}


\section{Introduction}
In this article, we are  concerned with Green's function for parabolic systems with Neumann boundary condition, which hereafter we shall refer to as Neumann Green's function.
It is often called Neumann heat kernel if the parabolic system has time-independent coefficients.
In a cylindrical domain $\Omega\times (-\infty,\infty)$, where the base $\Omega$ is an open connected set in $\bR^n$, we consider parabolic differential operator of the form
\begin{equation}		\label{eq1.01a}
\sL \vec u=  \frac{\partial \vec u}{\partial t}- \frac{\partial}{\partial x_\alpha} \left(A^{\alpha\beta}(x,t) \frac{\partial \vec u}{\partial x_\beta}\right).
\end{equation}
Here, we use the summation convention over repeated indices $1\le \alpha,\beta\le n$,  $\vec u$ is a column vector with components $u^1,\ldots, u^N$, $A^{\alpha\beta}(x,t)$ are $N\times N$ matrices whose elements $a^{\alpha\beta}_{ij}(x,t)$ are bounded measurable functions defined on the entire space $\bR^{n+1}$ such that for arbitrary $N\times n$ matrices $\vec \xi$, $\vec \eta$ with elements $\xi_\alpha^i$,  $\eta_\alpha^i$, we have 
\begin{equation}		\label{eq1.03a}
\lambda \abs{\vec \xi}^2 \le a^{\alpha\beta}_{ij} \xi^j_\beta \xi^i_\alpha,\qquad
\Abs{ a^{\alpha\beta}_{ij} \xi^j_\beta \eta^i_\alpha} \le \lambda^{-1} \abs{\vec \xi} \abs{\vec \eta},
\end{equation}
where $\lambda$ is positive constant and $\abs{\xi}^2=\xi_\alpha^i \xi_\alpha^i$.
We emphasize that we allow coefficients to be non-symmetric and time-dependent.
By a Neumann Green's function for the system \eqref{eq1.01a} in $\Omega\times (-\infty,\infty)$, we mean an $N \times N$ matrix valued function $\vec N(x,t,y,s)$ that is a solution of the initial-boundary value problem
\begin{equation*}
\left\{
\begin{array}{cl}
\sL_{x,t} \vec N(x,t,y,s) = 0 \qquad & \text{in }\; \Omega \times (s,\infty),\\
\frac{\partial}{\partial \nu_x} \vec N(x,t,y,s) =0 & \text{on }\; \partial\Omega \times (s,\infty),\\
\vec N(x,t,y,s) = \delta_y(x) \vec I &  \text{on }\; \Omega\times\set{t=s},
\end{array}
\right. 
\end{equation*}
where $\partial/\partial \nu$ is the natural co-normal derivative, $\delta_{y}(\cdot)$ is Dirac delta function concentrated at $y$,  $\vec I$ is the $N\times N$ identity matrix, and the equation as well as the boundary condition should be interpreted in some weak sense.
Equivalently, it can be defined as a function that weakly satisfies
\begin{equation*}
\left\{
\begin{array}{cl}
\sL_{x,t} \vec N(x,t,y,s)= \delta_y(x) \delta_{s}(t) \vec I \quad& \text{in }\; \Omega \times (-\infty,\infty),\\
\frac{\partial}{\partial \nu_x} \vec N(x,t,y,s) =0 \quad &\text{on }\; \partial\Omega \times (-\infty,\infty),\\
\vec N(x,t,y,s)\equiv 0\quad& \text{on }\; \Omega\times (-\infty,s).
\end{array}
\right. 
\end{equation*}
More precise definition of Neumann Green's function is given in Section~\ref{sec:ng}.
In the case when the coefficients are time-independent,  $\vec K(x,y,t)\coloneq \vec N(x,t,y,0)$ is usually called a Neumann heat kernel.

We prove that if the base $\Omega$ is such that the multiplicative embedding inequality \eqref{eq:mei} holds there and if weak solutions of the system \eqref{eq1.01a} satisfy an interior H\"older continuity estimate \eqref{IH}, then the Neumann Green's function exists and satisfies a natural growth estimate near the pole; see Theorem~\ref{EP.thm1}.
The multiplicative embedding inequality \eqref{eq:mei} holds if, for example,  $\Omega$ is a Sobolev extension domain with finite Lebesgue measure, a special Lipschitz domain, and an unbounded domain with compact Lipschitz boundary, etc.
Also, the H\"older continuity estimate \eqref{IH} holds, for example, in the scalar case (i.e., $N=1$), in the case when $n=2$ and the coefficients are time-independent, and in the case when the coefficients  belong to the VMO class; see Section~\ref{sec:apps}.

We are also interested in the following global Gaussian estimate for the Neumann Green's function:
For any $T>0$, there exists a positive constant $C$ such that for all $t,s$ satisfying $0<t-s<T$ and $x,y \in \Omega$, we have
\begin{equation}        \label{eq1.05a}
\abs{\vec N (x,t,y,s)} \le \frac{C}{(t-s)^{n/2}}\exp\Set{-\frac{\kappa\abs{x-y}^2}{t-s}},
\end{equation}
where $\kappa$ is a positive constant independent of $T$.
We present how one can derive the global estimate like \eqref{eq1.05a} using a local boundedness estimate \eqref{LB} for weak solutions of the  system with zero Neumann data; see Theorem~\ref{GE.thm1}.
It is well known that the local boundedness estimate \eqref{LB} is available, for example, when $N=1$ and $\Omega$ is a bounded extension domain or when $n=2$, the coefficients are time-independent, and $\Omega$ is a bounded Lipschitz domain, etc.; see Section~\ref{sec:apps}.
In fact, we show that the local boundedness estimate is a necessary and sufficient condition for the system to have a global Gaussian estimate for the Neumann Green's function; see Theorem~\ref{GE.thm2}.

The Green's function for a parabolic equation (i.e., $N=1$) with real measurable coefficients in the free space was first studied by Nash \cite{Nash} and its two-sided Gaussian bounds were later obtained by Aronson \cite{Aronson}.
There is a vast literature on Green's function for parabolic equations satisfying Dirichlet or Neumann conditions.
In the case when the coefficients are time-independent, they are called Dirichlet or Neumann heat kernel and have been studied by many authors; see Davies \cite{Davies89} and references therein, \cite{Robinson, VSC} for a treatment of heat kernels on non-Euclidean setting, and also a very recent article \cite{GSC11}.
It is well known that Aronson's bounds are no longer available for a parabolic equation with complex valued coefficient when $n\ge 3$.
Related to Kato square root problem, Auscher \cite{Auscher} obtained an upper Gaussian bound of the heat kernel for an elliptic operator whose coefficients are complex $L^\infty$-perturbation of real coefficients; see \cite{AMcT, AT, AT2} and also \cite{AtE, Daners, Ouhabaz05} for related results.
The Dirichlet Green's function for a parabolic system with time-dependent, measurable coefficients was studied in papers by Cho et al. \cite{CDK, CDK2}, where Dirichlet Green's function is constructed in any cylindrical domain under an assumption that the weak solutions of the system has an interior H\"older continuity estimate that is equivalent to ours.
In many respects, this article can be considered as natural follow-up of their papers.
However, we must point out it requires some effort and caution to handle Neumann condition and there are certainly some subtleties involved here.
The novelty of our paper lies in the following points.
First, we allow time-dependent, non-symmetric coefficients and do not specifically assume any regularity on them.
Secondly, we treat parabolic equations and systems simultaneously and do not rely on methods that are specific to single equations case, such as the maximum principle or Harnack's inequality.
Finally, our method separates existence of Neumann Green's function from its global estimates, which enables us to construct Neumann Green's functions in a large class of domains, where global Gaussian estimates for Neumann functions may not be available anymore.
In this paper, we only deal with so called strongly elliptic systems as it is described in \eqref{eq1.03a}, but our method can be as well applied to systems satisfying a weaker condition; see \cite{TKB1} for the Neumann heat kernel for the elliptic system of linear elasticity.
For the heat kernel for second-order elliptic operators in divergence form satisfying Robin-type boundary conditions, we bring attention to a very recent article \cite{GMN}.

The organization of the paper is as follows. In Section~\ref{sec:pre}, we introduce some notation and definitions including the precise definition of Neumann Green's function.
In Section~\ref{sec:main}, we state our three main theorems, which we briefly described above.
In Section~\ref{sec:apps}, we give examples for which our main results apply, and also in Theorem~\ref{thm4.4a}, we construct the Neumann function for second order elliptic systems with measurable coefficients in an extension domain $\Omega\subset \bR^2$ and obtain logarithmic bound for it.
We give the proofs for our main results and Theorem~\ref{thm4.4a} in Section~\ref{sec:pf}.

\section{Preliminaries}		\label{sec:pre}

\subsection{Basic notation and definitions}
We use $X=(x,t)$ to denote a point in $\bR^{n+1}$; $x=(x_1,\ldots, x_n)$ will always be a point in $\bR^n$.
We also write $Y=(y,s)$, $X_0=(x_0,t_0)$, etc.
We define the parabolic distance between the points $X=(x,t)$ and $Y=(y,s)$ in $\bR^{n+1}$ as
\[
\abs{X-Y}_\sP\coloneq \max(\abs{x-y}, \sqrt{\abs{t-s}}),
\]
where $\abs{\,\cdot\,}$ denotes the usual Euclidean norm.
We write $\abs{X}_\sP=\abs{X-0}_\sP$.
For an open set $Q \subset\bR^{n+1}$, we denote
\[
d_X=\dist(X,\partial_p Q)=\inf\bigset{\abs{X-Y}_\sP \colon Y\in \partial_p Q};\quad \inf \emptyset = \infty,
\]
where $\partial_p Q$ denotes the usual parabolic boundary of $Q$.
We use the following notions for basic cylinders in $\bR^{n+1}$:
\begin{align*}
Q^-_r(X)&=B_r(x) \times (t-r^2,t),\\
Q^+_r(X)&=B_r(x) \times (t,t+r^2),\\
Q_r(X)&= B_r(x)\times (t-r^2,t+r^2),
\end{align*}
where $B_r(x)$ is the usual Euclidean ball of radius $r$ centered at $x\in \bR^n$.
We use the notation
\[
\fint_Q u=\frac{1}{\abs{Q}}\int_Q u.
\]
We denote $a\wedge b=\min(a,b)$ and $a \vee b=\max(a,b)$.

\subsection{Function spaces}
For spaces of functions defined on a domain $\Omega\subset \bR^n$, we use the same notation as used in Gilbarg and Trudinger \cite{GT}, while for those defined on a domain $Q\subset \bR^{n+1}$, we borrow notation mainly from Ladyzhenskaya et al. \cite{LSU}.
To avoid confusion, spaces of functions defined on $Q\subset \bR^{n+1}$ will be always written in \emph{script letters} throughout the article.
For $q\ge 1$, we let $\sL_q(Q)$ denote the classical Banach space consisting of measurable functions on $Q$ that are $q$th-integrable.
$\sL_{q,r}(Q)$ is the Banach space consisting of all measurable functions on $Q$ with a finite norm
\[
\norm{u}_{\sL_{q,r}(Q)}=\left(\int_a^b\left(\int_{\Omega} \abs{u(x,t)}^q\, dx\right)^{r/q}dt\right)^{1/r},
\]
where $q\ge 1$ and $r\ge 1$.
$\sL_{q,q}(Q)$ will be denoted by $\sL_q(Q)$.
By $\sC^{\mu,\mu/2}(Q)$ we denote the set of all bounded measurable functions $u$ on $Q$ for which $\abs{u}_{\mu,\mu/2;Q}$ is finite, where
we define the parabolic H\"older norm as follows:
\begin{align*}
\abs{u}_{\mu,\mu/2;Q}&=[u]_{\mu,\mu/2;Q}+\abs{u}_{0;Q}\\
&\coloneq \sup_{\substack{X, Y \in Q\\ X\neq Y}}\frac{\abs{u(X)-u(Y)}}{\abs{X-Y}^\mu_\sP}+\sup_{X\in Q}\abs{u(X)}, \quad \mu\in(0,1].
\end{align*}
We write $u\in \sC^\infty_c(Q)$ (resp. $\sC^\infty_c(\bar Q)$) if $u$ is an infinitely differentiable function on $\bR^{n+1}$ with a compact support in $Q$ (resp. $\bar Q$).
We write $D_i u=D_{x_i} u=\partial u/\partial x_i$ ($i=1,\ldots,n$) and $u_t=\partial u /\partial t$.
We also write $Du=D_x u$ for the column vector $(D_1 u,\ldots, D_n u)^T$.
We write $Q(t_0)$ for the set of all points $(x,t_0)$ in $Q$ and $I(Q)$ for
the set of all $t$ such that $Q(t)$ is nonempty.
We denote
\[
\tri{u}_{Q}^2= \int_{Q}  \abs{D_x u}^2 \,dx \,dt+\esssup\limits_{t\in I(Q)} \int_{Q(t)} \abs{u(x,t)}^2\,dx.
\]
The space $\sW^{1,0}_q(Q)$ denotes the Banach space consisting of functions $u\in \sL_q(Q)$ with weak derivatives $D_i u \in \sL_q(Q)$ ($i=1,\ldots,n$) with the norm
\[
\norm{u}_{\sW^{1,0}_q(Q)}=\norm{u}_{\sL_q(Q)}+\norm{D_x u}_{\sL_q(Q)}
\]
and by $\sW^{1,1}_q(Q)$ the Banach space with the norm
\[
\norm{u}_{\sW^{1,1}_q(Q)}=\norm{u}_{\sL_q(Q)}+\norm{D_x u}_{\sL_q(Q)}+ \norm{u_t}_{\sL_q(Q)}.
\]
In the case when $Q$ has a finite height (i.e., $Q\subset \bR^{n}\times(-T,T)$ for some $T<\infty$), we define $\sV_2(Q)$ as the Banach space consisting of all elements of $\sW^{1,0}_2(Q)$ having a finite norm $\norm{u}_{\sV_2(Q)}\coloneq  \tri{u}_{Q}$ and the space $\sV^{1,0}_2(Q)$ is obtained by completing the set $\sW^{1,1}_2(Q)$ in the norm of $\sV_2(Q)$.
When $Q$ has an infinite height, we say that $u \in \sV_2(Q)$ (resp. $\sV^{1,0}_2(Q)$) if $u \in \sV_2(Q_T)$ (resp. $\sV^{1,0}_2(Q_T)$) for all $T>0$, where $Q_T=Q \cap \set{ \abs{t} <T}$, and $\tri{u}_Q <\infty$.
Note that this definition allows that $1\in \sV^{1,0}_2(\Omega\times(-\infty,\infty))$ when $\abs{\Omega}<\infty$.
Finally, we write $u\in \sL_{q,loc}(Q)$ if $u \in \sL_q(Q')$ for all $Q'\Subset Q$ and similarly define $\sW^{1,0}_{q,loc}(Q)$, etc.
\subsection{Weak solutions}
For $\vec f \in \sL_{2,1}(Q)^N$ and $\vec g^\alpha \in \sL_{2}(Q)^N$ ($\alpha=1,\ldots, n$), we say that $\vec u$ is a weak solution of $\sL \vec u=\vec f-D_\alpha \vec g^\alpha$ in $Q$ if $\vec u \in \sV_2(Q)^N$ and satisfies the identity
\begin{equation}		\label{eqn:E-71}
-\int_{Q} \vec u \cdot \vec \phi_t+ \int_{Q} \sB(\vec u, \vec \phi)= \int_{Q} \vec f \cdot \vec \phi+\int_{Q} \vec g^\alpha \cdot D_\alpha\vec \phi
\end{equation}
for all $\vec \phi \in \sC^\infty_c (Q)^N$, where we set
\[
\sB(\vec u,\vec v)= a^{\alpha\beta}_{ij} D_\beta u^j D_\alpha  v^i.
\]
The adjoint operator $\sL^{*}$ of  $\sL$ is given by
\begin{equation}		\label{eq1.01b}
\sL^{*} \vec u=  -\frac{\partial \vec u}{\partial t} -\frac{\partial}{\partial x_\alpha} \left(A^{\beta\alpha}(x,t)^T \frac{\partial \vec u}{\partial x_\beta}\right).
\end{equation}
Note that the coefficients of \eqref{eq1.01b} also satisfy the ellipticity and boundedness condition \eqref{eq1.03a} with the same constant $\lambda$.
We say that $\vec u$ is a weak solution of $\sL^{*} \vec u=\vec f-D_\alpha \vec g^\alpha$ in $Q$ if $\vec u \in \sV_2(Q)^N$ and satisfies, for all $\vec \phi \in \sC^\infty_c (Q)^N$, the identity
\begin{equation}		\label{eqn:E-71b}
\int_{Q} \vec u \cdot \vec \phi_t+ \int_{Q} \sB(\vec \phi, \vec u)= \int_{Q} \vec f \cdot \vec \phi+ \int_{Q} \vec g^\alpha \cdot D_\alpha\vec \phi.
\end{equation}
Let $\Omega$ be an open set in $\bR^n$ and $\Sigma$ be a relatively open subset of $\partial \Omega$.
Below we denote $Q= \Omega \times (a,b)$ and $S=\Sigma \times (a,b)$, where $-\infty \le a <b \le \infty$.
We say that $\vec u$ is a weak solution of
\[
\sL \vec u = \vec f-D_\alpha \vec g^\alpha \;\text{ in }\;Q,\quad \partial \vec u/\partial \nu = \vec g^\alpha \nu_\alpha \;\text{ on }\;S
\]
if $\vec u\in \sV_2(Q)^N$ and satisfies the identity \eqref{eqn:E-71} for all $\vec \phi \in \sC^\infty_c (Q\cup S)^N$.
Similarly, we say that $\vec u$ is a weak solution of
\[
\sL^{*} \vec u = \vec f-D_\alpha \vec g^\alpha\;\text{ in }\;Q,\quad \partial \vec u/\partial \nu^{*} =\vec g^\alpha \nu_\alpha \;\text{ on }\; S
\]
if $\vec u\in \sV_2(Q)^N$ and satisfies the identity \eqref{eqn:E-71b} for all $\vec \phi \in \sC^\infty_c (Q \cup S)^N$.
For $\vec \psi \in L^2(\Omega)^N$ and $a\neq -\infty$, we say that $\vec u$ is a weak solution in $\sV^{1,0}_2(Q)^N$ of the problem
\[
\left\{
\begin{aligned}
\sL\vec u=\vec f-D_\alpha \vec g^\alpha \quad &\text{in }\; Q\\
\partial \vec u/\partial \nu= \vec g^\alpha \nu_\alpha \quad&\text{on }\; \partial_p Q \\
\vec u(\cdot, a)=\vec \psi \quad&\text{on }\; \Omega
\end{aligned}
\right.
\]
if $\vec u \in \sV^{1,0}_2(Q)^N$ and satisfies the identity
\[
-\int_{Q} \vec u \cdot \vec \phi_t+ \int_{Q} \sB(\vec u, \vec \phi)- \int_{Q} \vec f \cdot \vec \phi-\int_{Q} \vec g^\alpha \cdot D_\alpha\vec \phi=\int_\Omega \vec \psi\cdot \vec \phi (\cdot,a)
\]
for all $\vec \phi \in \sC^\infty_c(\bar Q)^N$ that is equal to zero for $t=b$.
For $b\neq \infty$, we say that $\vec u$ is a weak solution in $\sV^{1,0}_2(Q)^N$ of the problem
\[
\left\{
\begin{aligned}
\sL^{*} \vec u=\vec f-D_\alpha \vec g^\alpha \quad &\text{in }\; Q\\
\partial \vec u/\partial \nu^{*}= \vec g^\alpha \nu_\alpha \quad&\text{on }\; \partial_p Q \\
\vec u(\cdot, b)=\vec \psi \quad&\text{on }\; \Omega.
\end{aligned}
\right.
\]
if $\vec u \in \sV^{1,0}_2(Q)^N$ and satisfies the identity
\[
\int_{Q} \vec u \cdot \vec \phi_t+ \int_{Q} \sB(\vec \phi, \vec u)- \int_{Q} \vec f \cdot \vec \phi- \int_{Q} \vec g^\alpha \cdot D_\alpha\vec \phi=\int_\Omega \vec \psi\cdot \vec \phi (\cdot,b)
\]
for all $\vec \phi \in \sC^\infty_c(\bar Q)^N$ that is equal to zero for $t=a$.

\subsection{Neumann Green's function}	\label{sec:ng}
Let  $\cQ \coloneq \Omega\times (-\infty,\infty)$, where $\Omega$ is a domain in $\bR^n$.
We say that an $N\times N$ matrix valued function $\vec N(X,Y)=\vec N (x,t,y,s)$, with measurable entries
\[
N_{ij} \colon \cQ\times \cQ \to  [-\infty, \infty],
\]
is Neumann Green's function of $\sL$ in $\cQ$ if it satisfies the following properties:
\begin{enumerate}[a)]
\item
For all $Y\in \cQ$,  we have $\vec N(\cdot,Y)\in \sW^{1,0}_{1,loc}(\cQ)^{N^2}$ and  $\vec N (\cdot,Y) \in \sV_{2}(\cQ \setminus Q_r(Y))^{N^2}$ for any $r>0$.

\item
For all $Y\in \cQ$, we have
$\sL \vec N(\cdot,Y) = \delta_Y \vec I$ in $\cQ$, $\frac{\partial}{\partial \nu} \vec N(\cdot,Y) =0$ on $\partial_p \cQ$, in the sense that for any $\vec \phi \in \sC^\infty_c(\bar \cQ)^N$, we have
\begin{equation}		\label{eq2.04x}
\int_{\cQ}\left( -N_{ik}(\cdot,Y)\phi^i_t+ a^{\alpha\beta}_{ij} D_\beta N_{jk}(\cdot,Y) D_\alpha \phi^i \right) = \phi^k(Y).
\end{equation}
\item
For any $\vec f=(f^1,\ldots, f^N)^T \in \sC^\infty_c(\bar \cQ)^N$, the function $\vec u$ given by
\[
\vec u(X)\coloneq \int_\cQ \vec N(Y,X)^T \vec f(Y)\,dY
\]
is a weak solution in $\sV^{1,0}_2(\cQ)^N$ of $\sL^{*} \vec u=\vec f$ in $\cQ$, $\partial \vec u/\partial \nu^{*} =0$ on $\partial_p \cQ$.
\end{enumerate}
We note that part c) of the above definition gives the uniqueness of a Neumann Green's function.

\section{Main results}			\label{sec:main}
We make the following assumptions to construct the Neumann Green's function in $\cQ \coloneq \Omega\times (-\infty,\infty)$.
\begin{enumerate}[{\bf{A}1.}]
\item
We assume that $C^\infty_c(\bar \Omega)$ is dense in $W^{1,2}(\Omega)$ and there exists a constant $\gamma$ such that for all $u \in \tilde{W}^{1,2}(\Omega)$, we have
\begin{equation}			\label{eq:mei}
\norm{u}_{L^{2(n+2)/n}(\Omega)} \le \gamma \norm{D u}_{L^2(\Omega)}^{n/(n+2)}  \norm{u}_{L^2(\Omega)}^{2/(n+2)},
\end{equation}
where we use the notation
\[
\tilde{W}^{1,2}(\Omega)\coloneq 
\begin{cases}
W^{1,2}(\Omega) &\mbox{if } \abs{\Omega}=\infty\\
\Set{u\in W^{1,2}(\Omega) \colon \int_\Omega u=0 } & \mbox{if }  \abs{\Omega}<\infty. 
\end{cases} 
\]
\item
There exist $\mu_0\in(0,1]$, $R_c \in (0,\infty]$, and $B_0>0$ such that if $\vec u$ is a weak solution of $\sL\vec u= 0$ (resp. $\sL^{*} \vec u=0$) in $Q=Q^{-}_R(X)$ (resp. $Q=Q^{+}_R(X)$),  where $X\in \cQ$ and $0<R< \bar d_X\coloneq d_X\wedge R_c$, then $\vec u$ is H\"older continuous in $\frac{1}{2}Q=Q^{-}_{R/2}(X)$ (resp. $\frac{1}{2}Q=Q^{+}_{R/2}(X)$) with an estimate
\begin{equation}	\label{IH}
[\vec u]_{\mu_0,\mu_0/2; \frac{1}{2}Q}\le B_0R^{-\mu_0}\left(\fint_{Q}\,\abs{\vec u(Y)}^2\,dY \right)^{1/2}.
\end{equation}
\end{enumerate}

\begin{remark}			\label{rmk3.3k}
In the case when $n\ge 3$, the inequality 
\begin{equation}			\label{eq3.04ktx}
\norm{u}_{L^{2n/(n-2)}(\Omega)} \le C \norm{D u}_{L^2(\Omega)}
\end{equation}
implies (via H\"older's inequality) the inequality \eqref{eq:mei} and also the inequality
\begin{equation}			\label{eq3.05ktx}
\norm{u}_{L^2(\Omega)} \le C \norm{D u}_{L^2(\Omega)}^{n/(n+2)} \norm{u}_{L^1(\Omega)}^{2/(n+2)}.
\end{equation}
If $\Omega=\bR^n$, then \eqref{eq3.04ktx} and \eqref{eq3.05ktx} are the Sobolev inequality and the Nash inequality respectively.
In fact, they are equivalent in that case; see the monograph by Saloff-Coste \cite{SC}.
From this point of view, one may think A1 roughly means that $\Omega$ behaves like $\bR^n$ with respect to $W^{1,2}$ functions.
\end{remark}
\begin{remark}
The inequality \eqref{eq:mei} implies the following:
If $u \in \sV_2(\Omega\times (a,b))$ is such that $u(\cdot, t) \in \tilde W^{1,2}(\Omega)$ for a.e. $t \in (a,b)$, where $-\infty\le a <b\le \infty$, then we have (see \cite[pp. 74--75]{LSU})
\begin{equation}			\label{eq5.13ddd}
\norm{u}_{\sL_{2(n+2)/n}(\Omega\times(a,b))} \le \gamma  \tri{u}_{\Omega\times(a,b)}.
\end{equation}
Another important consequence of \eqref{eq:mei} is that for $u\in \sV_2(\Omega\times(a,b))$ with $b-a<\infty$, we have
\begin{equation}			\label{paraemb}
\norm{u}_{\sL_{2(n+2)/n}(\Omega\times(a,b))} \le \left(2\gamma+(b-a)^{n/2}\abs{\Omega}^{-1}\right)^{\frac{1}{n+2}} \tri{u}_{\Omega\times(a,b)}.
\end{equation}
We refer to \cite[Eq.~(3.8), p. 77]{LSU} for the proof of \eqref{paraemb}. 
\end{remark}

\begin{theorem}		\label{EP.thm1}
Assume the conditions A1 and A2.
Then there exists a unique Neumann Green's function $\vec N(X,Y)=\vec N(x,t,y,s)$ of $\sL$ in $\cQ\coloneq \Omega\times (-\infty,\infty)$.
It is continuous in $\set{(X,Y) \in \cQ\times \cQ \colon X\neq Y}$ and satisfies $\vec N(x,t,y,s)=0$ for $t<s$.
For any $\vec f =(f^1,\ldots, f^N)^T \in \sC^\infty_c(\bar \cQ)^N$, the function $\vec u$ given by
\begin{equation}			\label{eq3.04c}
\vec u(X)\coloneq \int_{\cQ} \vec N(X,Y) \vec f(Y)\,dY
\end{equation}
is a weak solution in $\sV^{1,0}_2(\cQ)^N$ of
\begin{equation}			\label{eq3.05a}
\sL \vec u = \vec f\;\text{ in }\;\cQ,\quad \partial \vec u/\partial \nu =0\;\text{ on }\;\partial_p \cQ.
\end{equation}
Moreover, for all $\vec \psi =(\psi^1,\ldots,\psi^N)^T \in L^2(\Omega)^N$, the function given by
\begin{equation}		\label{EP.1b}
\vec u(x,t)=\int_\Omega \vec N (x,t,y,s) \vec \psi(y)\,dy
\end{equation}
is a unique weak solution in $\sV^{1,0}_2(\Omega\times(s,\infty))^N$ of the problem
\begin{equation}		\label{EP.1b1}			
\left\{
\begin{aligned}
\sL\vec u= 0\quad &\text{in }\; \Omega\times(s,\infty)\\
\partial \vec u/\partial \nu= 0 \quad&\text{on }\; \partial\Omega\times(s,\infty)\\
\vec u(\cdot, s)=\vec \psi \quad&\text{on }\; \Omega.
\end{aligned}
\right.
\end{equation}
Furthermore, for $X, Y \in \cQ$ satisfying $0<\abs{X-Y}_{\sP} < \frac{1}{2} \bar d_Y$, we have
\begin{equation}		\label{eq3.07c}
\abs{\vec N(X,Y)} \le C \abs{X-Y}_{\sP}^{-n},
\end{equation}
and for $X, X', Y\in \cQ$ satisfying $2\abs{X-X'}_{\sP}<\abs{X-Y}_{\sP}<\frac{1}{2} \bar d_Y$, 
\begin{equation}		\label{eq3.08d}
\abs{\vec N(X,Y)-\vec N(X',Y)}\le C\abs{X-X'}_{\sP}^{\mu_0} \abs{X-Y}_{\sP}^{-n-\mu_0},
\end{equation}
where $C=C(n,N,\lambda,\mu_0, B_0)$.
\end{theorem}

\begin{remark}			\label{rmk3.8}
In the proof of Theorem~\ref{EP.thm1} we will see that for any $Y\in \cQ$ and $0<R<\bar d_Y$, we have
\begin{enumerate}[{\em i)}]
\item
$\tri{{\vec N}(\cdot,Y)}_{\cQ\setminus \bar{Q}_R(Y)}\le CR^{-n/2}$.
\item
$\norm{\tilde{\vec N}(\cdot,Y)}_{\sL_{2(n+2)/n}(\cQ\setminus \bar Q_R(Y))}\le C_\gamma R^{-n/2}$, 
\item
$\bigabs{\bigset{X\in\cQ \colon \abs{\tilde{\vec N}(X,Y)}>\tau}} \le C_\gamma \tau^{-\frac{n+2}{n}}, \quad \forall \tau>\bar d_Y^{-n}$.
\item
$\bigabs{\bigset{X\in \cQ \colon \abs{D_x \vec N(X,Y)}>\tau}}\le C \tau^{-\frac{n+2}{n+1}},  \quad \forall \tau> \bar d_Y^{-(n+1)}$.
\item
$\norm{\vec N(\cdot,Y)}_{\sL_p(Q_R(Y))}\le C_{p,\gamma} R^{-n+(n+2)/p},  \quad \forall p\in [1,\tfrac{n+2}{n})$.
\item
$\norm{D_x \vec N(\cdot,Y)}_{\sL_p(Q_R(Y))} \le C_{p} R^{-n-1+(n+2)/p},  \quad \forall p\in[1,\tfrac{n+2}{n+1})$.
\end{enumerate}
Here, we denote $\tilde{\vec N}(\cdot,Y)=\vec N(\cdot,Y)-\frac{1}{\abs{\Omega}} H(t-s)\vec I$, where $H(t)=1_{[0,\infty)}(t)$ and we use the convention $1/\infty =0$.
\end{remark}

\begin{remark}
If $\Omega\subset \bR^n$ is such that there exists a bounded trace operator from $ W^{1,2}(\Omega)$ to $L^2(\Omega)$, then it can be shown that
for $\vec f \in \sC^\infty_c(\bar \Omega\times [a,b])^N$ and $\vec g \in \sC^\infty_c(\partial\Omega \times [a,b])^N$, where $-\infty<a<b<\infty$, $\vec u$ defined by
\[
\vec u(X)=\int_a^b\!\!\!\int_\Omega \vec N(X,Y)\vec f(Y)\,dY+\int_a^b\!\!\!\int_{\partial\Omega} \vec N(X,Y)\vec g(Y) \,dS_Y
\]
is a unique weak solution in $\sV^{1,0}_2(\Omega\times (a,b))^N$ of the problem
\[
\left\{
\begin{aligned}
\sL\vec u= \vec f\quad &\text{in }\; \Omega\times (a,b)\\
\partial \vec u/\partial \nu= \vec g \quad&\text{on }\; \partial\Omega\times(a,b)\\
\vec u(\cdot, a)=0 \quad&\text{on }\; \Omega.
\end{aligned}
\right.
\]
\end{remark}

The following ``local boundedness" assumption is used to obtain global Gaussian estimates for the Neumann Green's function.
\begin{enumerate}[{\bf{A}1.}]
\setcounter{enumi}{2}
\item
There exist $R_{M}\in (0,\infty]$, $B_1>0$ such that if $\vec u$ is a weak solution of
\begin{align*}
\sL \vec u = 0 \text{ in } Q=Q_R^{-}(X)\cap \cQ&,\quad \partial \vec u/\partial \nu =0 \text{ on } S=Q_R^{-}(X)\cap \partial_p \cQ,\\
(\text{resp. } \sL^{*} \vec u = 0 \text{ in } Q=Q_R^{+}(X)\cap \cQ&,\quad \partial \vec u/\partial \nu^{*} =0 \text{ on } S=Q_R^{+}(X)\cap \partial_p \cQ,)
\end{align*}
where $X\in \cQ$ and $0<R<R_M$, then $\vec u$ is bounded in $\frac{1}{2}Q= Q^{-}_{R/2}(X)\cap \cQ$ (resp.  $\frac{1}{2}Q= Q^{+}_{R/2}(X)\cap \cQ$) and we have an estimate
\begin{equation}		\label{LB}
\norm{\vec u}_{\sL_\infty(\frac{1}{2}Q)}\le B_1 R^{-(n+2)/2} \norm{\vec u}_{\sL_2(Q)}.
\end{equation}
\end{enumerate}

\begin{remark}
Note that A3 implies that there exists $\theta >0$ such that for all $x\in \Omega$ and $0<R<R_M$, we have
\begin{equation}		\label{ogr}
\theta R^n \le \abs{\Omega \cap B_R(x)},
\end{equation}
which follows from \eqref{LB} by taking $\vec u = \vec 1$ and $\theta=1/B_1^2$.
In particular, we should have $R_M <\infty$ in the case when $\abs{\Omega}<\infty$.
In fact, we may take $R_M=\abs{\Omega}^{1/n}$ in that case without loss of generality.
\end{remark}

\begin{theorem}		\label{GE.thm1}
Assume that the conditions A1-A3.
Then, for $t>s$, we have the Gaussian bound for the Neumann Green's function
\begin{equation}		\label{GE.thm1.1}
\abs{\vec N(x,t,y,s)}\le \frac{C}{\left\{(t-s)\wedge R_M^2\right\}^{n/2}}\exp\left(\frac{-\kappa\abs{x-y}^2}{t-s}\right),
\end{equation}
where $C=C(n, N, \lambda, B_1)$ and $\kappa=\kappa(\lambda)>0$.
In particular, the estimate \eqref{GE.thm1.1} implies that for $0<t-s \le T$, we have
\begin{equation}		\label{eq2.17dx}
\abs{\vec N(x,t,y,s)}\le \frac{C}{(t-s)^{n/2}}\exp\left(\frac{-\kappa\abs{x-y}^2}{t-s}\right),
\end{equation}
where $C=C(n,N,\lambda,B_1,T)$.
\end{theorem}

Finally, the next theorem says that the converse of Theorem \ref{GE.thm1} is also true, and thus that the condition A3 is equivalent to a global bound \eqref{GE.thm1.1} for the Neumann Green's function.
\begin{theorem}		\label{GE.thm2}
Assume that the conditions A1 and A2.
Suppose there exist constants $B_2$ and $\kappa>0$ such that the Neumann Green's function has the bound
\begin{equation}		\label{GE.thm2.1}	
\abs{\vec N(x,t,y,s)}\le  \frac{B_2}{\left\{(t-s)\wedge R_M^2\right\}^{n/2}} \exp\left(\frac{-\kappa\abs{x-y}^2}{t-s}\right),
\end{equation}
for all $t<s$ and $x, y\in\Omega$.
Then the condition A3 is satisfied with the same $R_M$ and $B_1$ depending on $n,N,\lambda,\kappa$, and $B_2$.
\end{theorem}

\section{Applications}			\label{sec:apps}
\subsection{Examples}
\subsubsection{Examples for A1}
First, we give examples of domains satisfying the condition A1.
\begin{enumerate}[(1)]
\item
We say that $\Omega$ is an extension domain (for $W^{1,2}$ functions) if there exists a linear operator $E \colon W^{1,2}(\Omega) \to W^{1,2}(\bR^n)$ such that
\begin{equation}		\label{extop}
\norm{Eu}_{L^2(\bR^n)} \le C \norm{u}_{L^2(\Omega)},\quad
\norm{Eu}_{W^{1,2}(\bR^n)} \le C \norm{u}_{W^{1,2}(\Omega)}.
\end{equation}
We recall that the multiplicative embedding inequality \eqref{eq:mei} holds for all $u \in W^{1,2}(\bR^n)$ with $n\ge 1$; see \cite[Theorem~2.2, p. 62]{LSU}.
If $\Omega\subset \bR^n$ ($n\ge 1$) is an extension domain satisfying $\abs{\Omega}<\infty$, then for $u\in \tilde W^{1,2}(\Omega)$, we have
\begin{multline*}
\norm{u}_{L^{2(n+2)/n}(\Omega)} \le \norm{Eu}_{L^{2(n+2)/n}(\bR^n)}\le C\norm{D(Eu)}_{L^2(\bR^n)}^{n/(n+2)} \norm{Eu}_{L^2(\bR^n)}^{2/(n+2)}\\
\le C\norm{u}_{W^{1,2}(\Omega)}^{n/(n+2)} \norm{u}_{L^2(\Omega)}^{2/(n+2)}\le C\norm{Du}^{n/(n+2)}_{L^2(\Omega)}\norm{u}^{2/(n+2)}_{L^2(\Omega)},
\end{multline*}
where we used the Poincar\'e's inequality in the last step, namely,
\[
\norm{u}_{L^2(\Omega)} \le C \norm{Du}_{L^2(\Omega)}.
\]
Therefore, we have A1 in this case when $\Omega\subset \bR^n$ ($n\ge 1$) is an extension domain in with finite Lebesgue measure.
Extension domains include locally uniform domains considered by Jones \cite{Jones};
the extension operator that originally appeared in \cite{Jones} may not satisfy the first inequality in \eqref{extop}, but Rogers \cite{Rogers} constructed an extension operator that satisfies both inequalities in \eqref{extop}.

\item
Let $\Omega$ be a special Lipschitz domain in $\bR^n$ ($n\ge 2$), i.e., a domain above a Lipschitz function $\varphi \colon \bR^{n-1}\to \bR$.
Then, there exists a linear extension operator $E \colon W^{1,2}(\Omega) \to W^{1,2}(\bR^n)$ such that (see \cite{EG})
\[
\norm{Eu}_{L^2(\bR^n)} \le C \norm{u}_{L^2(\Omega)},\quad
\norm{D(Eu)}_{L^{2}(\bR^n)} \le C \norm{Du}_{L^{2}(\Omega)}.
\]
Therefore, similar to the above, for any $u \in W^{1,2}(\Omega)$, we have
\begin{multline*}
\norm{u}_{L^{2(n+2)/n}(\Omega)} \le C\norm{D(Eu)}_{L^2(\bR^n)}^{n/(n+2)} \norm{Eu}_{L^2(\bR^n)}^{2/(n+2)}\\
\le C \norm{Du}_{L^2(\Omega)}^{n/(n+2)} \norm{u}_{L^2(\Omega)}^{2/(n+2)}.
\end{multline*}
More generally, if $\Omega=\bigcup_{i=1}^k \Omega_i$, where $\Omega_i$ is special Lipschitz domain (rotating the coordinates axes if necessary), then it also satisfies A1.
Indeed, 
\[
\norm{u}_{L^{2(n+2)/n}(\Omega_i)} \le C \norm{Du}_{L^2(\Omega_i)}^{n/(n+2)} \norm{u}_{L^2(\Omega_i)}^{2/(n+2)} \le C \norm{Du}_{L^2(\Omega)}^{n/(n+2)} \norm{u}_{L^2(\Omega)}^{2/(n+2)},
\]
and thus we have
\[
\norm{u}_{L^{2(n+2)/n}(\Omega)} \le Ck \norm{Du}_{L^2(\Omega)}^{n/(n+2)} \norm{u}_{L^2(\Omega)}^{2/(n+2)}.
\]

\item
Suppose that $\Omega$ is an unbounded domain in $\bR^n$ ($n\ge 3$) with compact boundary and that $\Omega$ has the cone property.
Then, for all $u \in W^{1,2}(\Omega)$, we have (see \cite[Theorem~1]{CWZ})
\[
\norm{u}_{L^{2n/(n-2)}(\Omega)}\le C\norm{Du}_{L^2(\Omega)},
\]
which together with H\"older's inequality imply A1; see Remark~\ref{rmk3.3k}.
\end{enumerate}

\subsubsection{Examples for A2}
Next, we give examples when the condition A2 holds.
\begin{enumerate}[(1)]
\item
In the scalar case ($N=1$), De Giorgi-Moser-Nash theory implies A2 with $\mu_0=\mu_0(n,\lambda)$, $R_c=\infty$, and  $B_0=B_0(n,\lambda)$.
\item
If $n=2$ and the coefficients are time-independent, then A2 holds with $\mu_0=\mu_0(\lambda)$, $R_c=\infty$, and $B_0=B_0(N,\lambda)$.
It is a consequence of the well known theorem of Morrey \cite[pp. 143--148]{Morrey}; see \cite[Theorem~3.3]{Kim} for the derivation of A2 from Morrey's estimate.
\item
Let $\bar A^{\alpha\beta}$ be scalar functions on $\bR^{n+1}$ such that there is a constant $\lambda_0>0$ such that
\[
\bar A^{\alpha\beta}\xi_\beta\xi_\alpha\ge \lambda_0\abs{\vec \xi}^2,\quad\Abs{\bar A^{\alpha\beta} \xi_\beta \eta_\alpha} \le \lambda_0^{-1} \abs{\vec \xi} \abs{\vec \eta},
\]
for arbitrary column vectors $\vec \xi$, $\vec \eta \in \bR^n$ with elements $\xi_\alpha$,  $\eta_\alpha$.
We denote
\[
\mathscr{E}= \sup_{X\in \bR^{n+1}}\left\{
 \sum_{i,j=1}^N \sum_{\alpha,\beta =1}^n \Abs{a^{\alpha\beta}_{ij}(X)-\bar A^{\alpha\beta}(X)\delta_{ij}}^2\right\}^{1/2}.
\]
Then, there exists $\mathscr{E}_0=\mathscr{E}_0(n,\lambda_0)$ such that if $\mathscr{E}<\mathscr{E}_0$, then A2 is satisfied with $\mu_0=\mu_0(n,\lambda_0)$, $R_c=\infty$, and $B_0=B_0(n,N,\lambda_0)$; see \cite[Lemma~2.2]{CDK}.

\item
For a measurable function $f$ defined on $\bR^{n+1}$, we set for $\rho>0$
\begin{multline*}
\omega_\rho=\omega_\rho(f)=\sup_{X\in\bR^{n+1}}\sup_{r \leq \rho} \fint_{t-r^2}^{t+r^2}\!\!\!\fint_{B_r(x)} \bigabs{f(y,s)-\bar f_{x,r}(s)}\,dy\,ds\\
\text{where }\;\bar f_{x,r}(s)=\fint_{B_r(x)}f(y,s)\,dy.
\end{multline*}
We say that $f$ belongs to $\VMO_x$ if $\lim_{\rho\to 0} \omega_\rho(f)=0$.
If the coefficients belong to $\VMO_x$ class (or their $\BMO_x$ modulus $\omega_\rho$ are sufficiently small) then A2 holds with $R_c$, and $\mu_0, B_0$ depending on $n, N, \lambda$, $\Omega$ and the $\BMO_x$ modulus $\omega_\rho$ of the coefficients.
If $\Omega$ is bounded, then we can take $R_c=\diam\Omega$; see \cite[Lemma~2.3]{CDK}.
\end{enumerate}

\subsubsection{Examples for A3}			\label{sec:4.1.3}
Finally, we give examples when the condition A3 holds.
\begin{enumerate}[(1)]
\item
Let $N=1$ and $\Omega$ be an extension domain for which there exists $\theta>0$ and $R_M\in (0,\infty]$ such that \eqref{ogr} holds for all $x\in\Omega$ and $0<R<R_M$.
Then, A3 holds with $R_M$ and $B_1=B_1(n,\lambda, \Omega, \theta)$.
This follows from the De Giorgi-Moser-Nash theorem.
\item
If $n=2$, the coefficients are time-independent, and $\Omega$ is a bounded Lipschitz domain or a special Lipschitz domain, then A3 holds with $R_M=\diam\Omega$ and $B_1=B_1(N,\lambda, \Omega)$.
The proof is essentially the same as that of \cite[Lemma~4.4]{DK09}.

\item
If the coefficients belong to $\VMO_x$ class (or with sufficiently small $\BMO_x$ modulus $\omega_\rho$) and $\Omega$ is bounded $C^1$ domain (or Lipschitz domain with small Lipschitz constant), then A3 holds with $R_M=\diam\Omega$ and $B_1$ depending on $n, N, \lambda, \Omega$, and $\BMO_x$ modulus $\omega_\rho$ of the coefficients; see \cite[Lemma~6.1]{CK2} and \cite[Corollary~4.11]{CDK2}.
\end{enumerate}

\subsection{Elliptic Neumann function}			\label{sec:4.2enf}
Let us consider elliptic differential operator of the form
\[
L \vec u=  -\frac{\partial}{\partial x_\alpha} \left(A^{\alpha\beta}(x) \frac{\partial \vec u}{\partial x_\beta}\right),
\]
where $A^{\alpha\beta}(x)$ are $N\times N$ matrices whose elements $a^{\alpha\beta}_{ij}(x)$ are bounded measurable functions satisfying \eqref{eq1.03a}, and its adjoint operator $L^{*}$ defined by
\[
L^{*} \vec u=  -\frac{\partial}{\partial x_\alpha} \left(A^{\beta\alpha}(x)^T \frac{\partial \vec u}{\partial x_\beta}\right).
\]
Let $\Omega\subset \bR^n$ be an extension domain such that $\abs{\Omega}<\infty$.
We consider Neumann boundary value problem
\begin{equation}	\label{NP}
\left\{
\begin{aligned}
L\vec u&=\vec f \quad\text{in }\;\Omega,\\
\partial \vec u/\partial \nu&= 0\quad \text{on }\;\partial\Omega.
\end{aligned}
\right.
\end{equation}
Given $\vec f \in (W^{1,2}(\Omega)^{*})^N$, we say that $\vec u$ is a weak solution in $\tilde W^{1,2}(\Omega)^N$ of the problem \eqref{NP} if we have $\vec u \in \tilde{W}^{1,2}(\Omega)^N$ and
\[
\sB(\vec u,\vec \phi) = \ip{\vec f,\vec \phi}, \quad \forall \vec \phi \in \tilde W^{1,2}(\Omega)^N.
\]
By the standard elliptic theory, we have a unique weak solution in $\tilde W^{1,2}(\Omega)^N$ of the problem \eqref{NP} provided $\vec f$ satisfies the \emph{compatibility condition} $\ip{\vec f, \vec 1}=0$.
We say that an $N\times N$ matrix valued function $\vec G(x,y)$ is the  Neumann function of $L$ in $\Omega$ if it satisfies the following properties:
\begin{enumerate}[i)]
\item
$\vec G(\cdot,y) \in W^{1,1}_{loc}(\Omega)^{N^2}$ and $\vec G(\cdot,y) \in W^{1,2}(\Omega\setminus B_r(y))^{N^2}$ for all $y\in\Omega$ and $r>0$.
Moreover, $\int_{\Omega} \vec G(x,y)\,dx=0$.
\item
$L\vec G(\cdot,y)=(\delta_y-\frac{1}{\abs{\Omega}}) \vec I$ in $\Omega$,  $\partial \vec G(\cdot,y)/\partial \nu=0$ on $\partial\Omega$ for all $y\in\Omega$ in the  sense that for $\vec \phi \in C^\infty_c(\bar \Omega)^N$ satisfying $\int_\Omega \vec \phi=0$, we have
\[
\int_{\Omega}a^{\alpha\beta}_{ij} D_\beta G_{jk}(\cdot,y) D_\alpha \phi^i = \phi^k(y).
\]
\item
For any $\vec f \in C_c^\infty(\bar\Omega)^N$ satisfying $\int_\Omega \vec f=0$, the function $\vec u$ defined by
\begin{equation}	\label{choej}
\vec u(x)=\int_\Omega \vec G(y,x)^T \vec f(y)\,dy
\end{equation}
is the weak solution in $\tilde{W}^{1,2}(\Omega)^N$ of the problem
\begin{equation}		\label{hee}
\left\{
\begin{aligned}
L^{*}\vec u&=\vec f \quad\text{in }\;\Omega,\\
\partial \vec u/\partial \nu^{*}&= 0\quad \text{on }\;\partial\Omega.
\end{aligned}
\right.
\end{equation}
\end{enumerate}
Note the property iii) and the requirement $\int_\Omega \vec G(x, y)\,dx=0$ give the uniqueness of the Neumann function.
We warn the reader that the above definition is slightly different from the one given in a recent article \cite{CK2}, where it is assumed that $n\ge 3$ and $\partial\Omega$ is Lipschitz so that the boundary trace of a $W^{1,2}(\Omega)$ function is well defined.
In \cite{CK2}, the Neumann function satisfies the \emph{normalization condition} $\int_{\partial\Omega} \vec G(x,y)\,dS_x =0$ instead of $\int_\Omega \vec G(x,y)\,dx=0$.
We point out that they differ by a function $\vec U$, which is a weak solution in $W^{1,2}(\Omega)^{N^2}$ of the problem
\[
\left\{
\begin{aligned}
L\vec U&=\tfrac{1}{\abs{\Omega}}\vec I \quad\text{in }\;\Omega,\\
\partial \vec U/\partial \nu&= -\tfrac{1}{\abs{\partial\Omega}}\vec I\quad \text{on }\;\partial\Omega.
\end{aligned}
\right.
\]
We treat only the two dimensional case below because the case when $n\ge 3$ is studied in \cite{CK2} extensively by a different but more direct method.
We also mention that in a recent article \cite{TKB2}, the Neumann function is constructed in $\BMO(\Omega)$, where $\Omega$ is a bounded Lipschitz domain in $\bR^2$.

\begin{theorem}		\label{thm4.4a}
Let $\Omega\subset \bR^2$ be an extension domain such that $\abs{\Omega}<\infty$.
Then, there exists the Neumann function $\vec G(x,y)$ of $L$ in $\Omega$.
It is continuous in $\set{(x,y)\in\Omega\times\Omega \colon x\neq y}$ and for $x,y$ satisfying $0<\abs{x-y}<R/2$, where $R=d_x \wedge d_y$,  we have
\begin{equation}			\label{eq4.03mr}
\abs{\vec G(x,y)} \le C\left(1+\varrho^2 R^{-2}+ \ln (R/\abs{x-y})\right),
\end{equation}
where $C=C(N,\lambda)$ and $\varrho$ is the same constant as in \eqref{eq002v}.
Also, we have
\begin{equation}
\label{eq:E23}
\vec G(y,x)= \vec G^{*}(x,y)^T
\end{equation}
where $\vec G^{*}(x,y)$ is the Neumann function of the adjoint $L^{*}$ in $\Omega$.
Moreover, if $\Omega$ is a Lipschitz domain, then for all $x, y\in\Omega$, we have
\begin{equation}		\label{eq13.03k}
\abs{\vec G(x,y)} \le C\left(1+\ln (d/\abs{x-y})\right),\quad d=\diam \Omega,
\end{equation}
where $C$ depends on the Lipschitz character of $\partial \Omega$, $\diam \Omega$ as well as $N$ and $\lambda$.
\end{theorem}

\section{Proofs of main Theorems}		\label{sec:pf}

\subsection{Proof of Theorem~\ref{EP.thm1}: Case when $\abs{\Omega}<\infty$.}\label{sec5.1}
In the proof, we denote by $C$ a constant depending on $n, N, \lambda,  \mu_0, B_0$ unless otherwise stated; if it depends also on $\gamma$, it will be written as $C_\gamma$, etc.
We begin with constructing mollified Neumann Green's functions.
Fix a function $\Phi\in C^\infty_c(B_1(0))$ satisfying $0\le \Phi\le 2$ and $\int_{\bR^n}\Phi=1$.
Let $Y=(y,s)\in \cQ$ be fixed but arbitrary.
For $\epsilon>0$, we define
\[
\Phi_{y,\epsilon}(x)=\epsilon^{-n}\Phi((x-y)/\epsilon).
\]
For any $T>s$, let $\vec v_\epsilon=\vec v_{\epsilon;Y,k}$ be a unique weak solution in $\sV^{1,0}_2(\Omega\times(s,T))^N$ of the problem
\begin{equation}		\label{mn.eq1}
\left\{
\begin{aligned}
\sL\vec u= 0\quad &\text{in }\; \Omega\times(s,T)\\
\partial \vec u/\partial \nu=  0 \quad&\text{on }\; \partial\Omega\times (s,T)\\
\vec u= \Phi_{y,\epsilon}\vec e_k \quad&\text{on }\; \Omega\times \set{s},
\end{aligned}
\right.
\end{equation}
where $\vec e_k$ is the $k$-th unit vector in $\bR^N$; with an aid of the inequality \eqref{paraemb}, which was a consequence of the condition A1, the unique solvability of the problem \eqref{mn.eq1} follows from the Galerkin method described  in \cite[\S III.5]{LSU}.
By setting $\vec v_\epsilon(x,t)=0$ for $t<s$ and letting $T\to \infty$, we may assume the $\vec v_\epsilon$ is defined on the entire $\cQ$.
Then, by the energy inequality, we obtain
\begin{equation}		\label{mn.eq1c}
\tri{\vec v_\epsilon}_{\cQ} \le C\norm{\Phi_{y,\epsilon}}_{L^2(\Omega)}\le C \epsilon^{-n/2}.
\end{equation}
Thus $\vec v_\epsilon$ belongs to $\sV_2(\cQ)^N$.
We define the \emph{mollified Neumann Green's function} $\vec N^\epsilon(\cdot,Y)=\big( N^\epsilon_{jk}(\cdot,Y)\big)^N_{j,k=1}$ for $\sL$ by setting
\[
N^\epsilon_{jk}(\cdot,Y)=v_\epsilon^j=v^j_{\epsilon;Y,k}.
\]
Next, for a given $\vec f=(f^1,\ldots,f^N)^T \in \sC^\infty_c(\bar \Omega\times (a,b))^N$, where $a<s<b$, let $\vec u$ be a unique weak solution in $\sV^{1,0}_2(\Omega\times (a,b))^N$ of the \emph{backward problem}
\begin{equation}		\label{mn.eq2}		
\left\{
\begin{aligned}
\sL^{*}\vec u=\vec f\quad &\text{in }\; \Omega\times (a,b)\\
\partial \vec u/\partial \nu^{*}=  0 \quad&\text{on }\; \partial\Omega\times (a,b)\\
\vec u=  0 \quad&\text{on }\; \Omega\times \set{b}.
\end{aligned}
\right.
\end{equation}
By setting $\vec u(x,t)=0$ for $t>b$ and letting $a\to -\infty$, we may assume that $\vec u$ is defined on the entire $\cQ$.
Then, it is easy to see that we have the identity
\begin{equation}		\label{mn.eq3}
\int_\Omega \Phi_{y,\epsilon}(x) u^k(x,s)\,dx =\iint_{\cQ} N^\epsilon_{ik}(X,Y) f^i(X)\,dX.
\end{equation}
If we assume that $\vec f$ is supported in $Q= Q^{+}_R(X_0)$, where $0<R<\bar d_Y$, then the energy inequality and H\"older's inequality yields
\[
\tri{\vec u}^2_{\cQ} \le C \norm{\vec f}_{\sL_{2(n+2)/(n+4)}(Q)}\norm{\vec u}_{\sL_{2(n+2)/n}(Q)}.
\]
Then, by the inequality \eqref{paraemb} applied $Q=Q^{+}_R(X_0)=B_R(x_0)\times (t_0,t_0+R^2)$, we get
\begin{equation}		\label{mn.eq4b}
\tri{\vec u}_{\cQ}\le C\norm{\vec f}_{\sL_{2(n+2)/(n+4)}(Q)}.
\end{equation}
By utilizing \eqref{mn.eq4b} and the condition A2, and proceeding as in \cite[Section~3.2]{CDK}, we find that $\vec u$ is continuous in $\frac{1}{4}Q=Q_{R/4}^{+}(X_0)$ and satisfies (see \cite[Eq.(3.15)]{CDK})
\begin{equation}		\label{mn.eq4d}
\abs{\vec u}_{0; \frac{1}{4}Q} \le CR^{2}\norm{\vec f}_{\sL_\infty(Q)}.
\end{equation}
If $B_\epsilon(y)\times \set{s}\subset \frac{1}{4}Q$, then by \eqref{mn.eq3} and \eqref{mn.eq4d}, we obtain
\[
\Abs{\int_{Q} N^\epsilon_{ik}(\cdot,Y)f^i}\le \abs{\vec u}_{0; \frac{1}{4}Q}\le CR^2\norm{\vec f}_{\sL_\infty(Q)}.
\]
Therefore, by duality, it follows that we have
\begin{equation}		\label{mn.eq4e}
\Norm{\vec N^\epsilon(\cdot,Y)}_{\sL_1(Q)}\le CR^2,
\end{equation}
provided $0<R<\bar d_Y$ and $B_\epsilon(y)\times \set{s}\subset \frac{1}{4}Q=Q_{R/4}^{+}(X_0)$.
For $X\in \cQ$ such that $0<d\coloneq \abs{X-Y}_\sP< \bar d_Y/6$, if we set $r=d/3$, $X_0=(y, s-2d^2)$, and $R=6d$, then it is easy to see that for $\epsilon<d/3$, we have
\[
B_\epsilon(y)\times\set{s}\subset Q^+_{R/4}(X_0), \quad Q^{-}_r(X)\subset Q^{+}_R(X_0 ),
\]
and also that $\vec v_\epsilon=\vec v_{\epsilon;Y,k}$ satisfies $\sL\vec v_\epsilon=0$ in $Q^{-}_r(X)$.
Then, we derive from the condition~A2 that (see \cite[Lemma~2.6]{CDK})
\begin{equation}		\label{mn.eq4g}
\norm{\vec v_\epsilon}_{\sL_\infty(Q_{r/2}^{-}(X))}\le \frac{C}{r^{n+2}}\norm{\vec v_\epsilon}_{\sL_1(Q^{-}_r(X))}.
\end{equation}
By combining \eqref{mn.eq4e} and \eqref{mn.eq4g}, we obtain that
\begin{equation*}
\abs{\vec v_\epsilon(X)}\le C r^{-n-2}\norm{\vec v_\epsilon}_{\sL_1(Q_r^-(X))}\le Cr^{-n}.
\end{equation*}
That is, for any $X,Y\in \cQ$ satisfying $0<\abs{X-Y}_\sP<\bar d_Y/6$, we get
\begin{equation}		\label{mn.eq4h}
\abs{\vec N^\epsilon(X,Y)}\le C\abs{X-Y}^{-n}_\sP, \quad  \forall \epsilon<\tfrac{1}{3}\abs{X-Y}_\sP.
\end{equation}
Next, notice that we have
\[
\int_{\Omega}\vec N^\epsilon(x,t,y,s)\,dx=H(t-s)\vec I;\quad H(t)=1_{[0,\infty)}(t).
\]
We define
\begin{equation}		\label{wse.eq1aa}
\tilde{\vec N}{}^\epsilon(x,t,y,s)=\vec N^\epsilon(x,t,y,s)-\frac{1}{\abs{\Omega}}H(t-s)\vec I
\end{equation}
so that
\begin{equation}		\label{eq5.12c}
\int_\Omega\tilde{\vec N}{}^\epsilon(x,t,y,s)\,dx= 0.
\end{equation}
Therefore, by \eqref{eq5.13ddd}, for any $-\infty\le a<b\le \infty$, we have
\begin{equation}			\label{eq5.13d}
\norm{\tilde{\vec N}{}^\epsilon(\cdot,Y)}_{\sL_{2(n+2)/n}(\Omega\times(a,b))} \le \gamma  \tri{\tilde{\vec N}{}^\epsilon(\cdot,Y)}_{\Omega\times(a,b)}.
\end{equation}

\begin{lemma}		\label{wse.lem1}
For any $Y\in \cQ$, $0<R<\bar d_Y$, and $\epsilon>0$, we have
\begin{gather}		
\label{wse.eq2b}
\tri{\tilde{\vec N}{}^\epsilon(\cdot,Y)}_{\cQ\setminus \bar{Q}_R(Y)}\le CR^{-n/2},\\
\label{wse.eq2c}
\norm{\tilde{\vec N}{}^\epsilon(\cdot,Y)}_{\sL_{2(n+2)/n}(\cQ\setminus \bar Q_R(Y))}\le C_\gamma R^{-n/2}.
\end{gather}
Also, for any $Y\in \cQ$ and $\epsilon>0$, we have
\begin{align}
\label{wse.eq2e}
\bigabs{\bigset{X\in\cQ \colon \abs{\tilde{\vec N}{}^\epsilon(X,Y)}>\tau}} &\le C_\gamma \tau^{-\frac{n+2}{n}}, \quad \forall \tau>\bar d_Y^{-n},\\
\label{wse.eq2d}
\bigabs{\bigset{X\in \cQ \colon \abs{D_x\tilde{\vec N}{}^\epsilon(X,Y)}>\tau}}&\le C \tau^{-\frac{n+2}{n+1}},  \quad \forall \tau>\bar d_Y^{-(n+1)}.
\end{align}
Furthermore, for any $Y\in \cQ$, $0<R<\bar d_Y$, and $\epsilon>0$, we have
\begin{align}
\label{wse.eq2h}
\norm{\tilde{\vec N}{}^\epsilon(\cdot,Y)}_{\sL_p(Q_R(Y))}&\le C_{p,\gamma} R^{-n+(n+2)/p},  \quad \forall p\in [1,\tfrac{n+2}{n}),\\
\label{wse.eq2i}
\norm{D\tilde{\vec N}{}^\epsilon(\cdot,Y)}_{\sL_p(Q_R(Y))} &\le C_{p} R^{-n-1+(n+2)/p},  \quad \forall p\in[1,\tfrac{n+2}{n+1}).
\end{align}
\end{lemma}

\begin{proof}
By \eqref{mn.eq1c}, \eqref{mn.eq4h}, and the energy inequality, we get (see \cite[Eq.~(3.20)]{CDK})
\[
\tri{\vec N^\epsilon(\cdot,Y)}_{\cQ\setminus \bar{Q}_R(Y)}\le CR^{-n/2},\quad 0<R<\bar d_Y/6.
\]
from which and \eqref{wse.eq2b} follows readily since $\abs{\Omega} \ge C R^n$ and $\bar d_Y$ is comparable to $\bar d_Y/6$.
In fact, if we fix a $\zeta \in \sC^\infty_c(Q_R(Y))$, where $0<R<\bar d_Y/6$, such that
\begin{equation}		\label{wse.eq3aa}
0\le \zeta\le1, \quad \zeta\equiv 1 \text{ on }Q_{R/2}(Y), \quad \abs{D_x \zeta}^2+ \abs{\zeta_t}\le 16 R^{-2},
\end{equation}
then we have (see \cite[Eq.~(3.22)]{CDK})
\[
\tri{(1-\zeta)\vec N^\epsilon(\cdot,Y)}_{\cQ}\le CR^{-n/2},
\]
and thus by \eqref{wse.eq1aa} and the properties of $\zeta$, we get
\begin{equation}
\label{wse.eq3c}
\tri{(1-\zeta)\tilde{\vec N}{}^\epsilon(\cdot,Y)}_{\cQ}\le C(R^{-n/2}+R^{n/2} \abs{\Omega}^{-1}+\abs{\Omega}^{-1/2}) \le C R^{-n/2}.
\end{equation}
To get \eqref{wse.eq2c}, set $\cQ_{(1)}\coloneq \Omega\times(s+R^2,\infty)$ and $\cQ_{(2)}\coloneq (\Omega\setminus B_R(y))\times(s,s+R^2)$ and note that
by \eqref{eq5.13d} and \eqref{wse.eq2b} we have
\begin{equation}		\label{wse.eq3e}
\norm{\tilde{\vec N}{}^\epsilon(\cdot,Y)}_{\sL_{2(n+2)/n}(\cQ_{(1)})} \le C \gamma \tri{\tilde{\vec N}{}^\epsilon(\cdot,Y)}_{\cQ_{(1)}}\le C\gamma R^{-n/2}
\end{equation}
and similarly, by \eqref{paraemb} and \eqref{wse.eq3c}, we have
\begin{multline}	\label{wse.eq3d}
\norm{\tilde{\vec N}{}^\epsilon(\cdot,Y)}_{\sL_{2(n+2)/n}(\cQ_{(2)})} \le (2\gamma+1)^{\frac{1}{n+2}}\tri{(1-\zeta)\tilde{\vec N}{}^\epsilon(\cdot,Y)}_{\Omega\times(s,s+R^2)} \\
\le C (2\gamma+1)^{\frac{1}{n+2}}  R^{-n/2}.
\end{multline}
Therefore, by combining \eqref{wse.eq3e} and \eqref{wse.eq3d}, we obtain \eqref{wse.eq2c}.
We derive \eqref{wse.eq2h} and \eqref{wse.eq2i}, respectively, from \eqref{wse.eq2e} and  \eqref{wse.eq2d}, which in turn follows respectively from  \eqref{wse.eq2c} and \eqref{wse.eq2b}; see \cite[Lemmas~3.3 and 3.4]{CDK}.
\end{proof}

\begin{lemma}
Let $\set{u_k}_{k=1}^\infty$ be a sequence in $\sV_2(\cQ)$.
If $\sup_k \tri{u_k}_{\cQ} \le A$, then there exist a subsequence $\set{u_{k_j}}_{j=1}^\infty \subseteq \set{u_k}_{k=1}^\infty$ and $u\in \sV_2(\cQ)$ with $\tri{u}_{\cQ}\le A$ such that $u_{k_j} \rightharpoonup u$ weakly in  $\sW^{1,0}_2(\Omega\times (a,b))$ for all $-\infty<a<b<\infty$.
\end{lemma}
\begin{proof}
See \cite[Lemma~A.1]{CDK}.
\end{proof}

The above two lemmas contain all the ingredients for the construction of a function $\tilde{\vec N}(\cdot,Y)$ such that for a sequence $\epsilon_\mu$ tending to zero, we have
\begin{align*}
\tilde{\vec N}{}^{\epsilon_\mu}(\cdot,Y) &\rightharpoonup \tilde{\vec N}{}(\cdot,Y) \;\text{ weakly in }\, \sW^{1,0}_q(Q_{\bar d_Y}(Y))^{N^2},\\
(1-\zeta)\tilde{\vec N}{}^{\epsilon_\mu}(\cdot,Y) &\rightharpoonup (1-\zeta)\tilde{\vec N}(\cdot,Y) \; \text{ weakly in }\,\sW^{1,0}_2(\Omega\times (-T,T))^{N^2}, \;\; \forall T>0,
\end{align*}
where $1<q<\frac{n+2}{n+1}$ and $\zeta$ is as defined \eqref{wse.eq3aa} with $R=\bar d_Y/2$.
It is routine to verify that $\tilde{\vec N}(\cdot,Y)$ satisfies the same estimates as in Lemma~\ref{wse.lem1}; see \cite[Sec.~4.2]{CDK}.
Note that \eqref{eq5.12c} implies $\int_\Omega \tilde{\vec N}(x,t,y,s)\,dx =0$.
We define $\vec{N}(\cdot,Y)$ by
\begin{equation}	\label{ntilde}
\vec N(x,t,y,s)=\tilde{\vec N}(x,t,y,s)+\frac{1}{\abs{\Omega}}H(t-s)\vec I.
\end{equation}
Then, by the previous remark that it is clear that $\vec N(X,Y)$ satisfies the local estimates in Remark~\ref{rmk3.8}, and thus the property a) for the Neumann Green's function defined in Sec.~\ref{sec:ng}.
We shall now show that $\vec N(X,Y)$ also satisfies the properties b) and c) so that $\vec N(X,Y)$ is indeed the Neumann Green's function.
To verify the property b), let us assume $\vec \phi=(\phi^1,\ldots, \phi^N)^T$ is supported in $\bar\Omega \times (-T,T)$, where $-T<s<T$, and  note that
\begin{align}
\label{eq5.27a}
\vec N^{\epsilon_\mu}(\cdot,Y) &\rightharpoonup \vec N(\cdot,Y) \;\text{ weakly in }\, \sW^{1,0}_q(Q_{\bar d_Y}(Y))^{N^2},\\
\label{eq5.28b}
(1-\zeta)\vec N^{\epsilon_\mu}(\cdot,Y) &\rightharpoonup (1-\zeta) \vec N(\cdot,Y) \; \text{ weakly in }\,\sW^{1,0}_2(\Omega\times (-T,T))^{N^2}.
\end{align}
Since $k$th column of $\vec N^\epsilon(\cdot,Y)$ is the weak solution in $\sV^{1,0}_2(\Omega\times(s,T))^N$ of the problem \eqref{mn.eq1}, we find that
\begin{multline*}
\int_\Omega \Phi_{\epsilon_\mu,y}(x) \phi^k(x)\,dx\\=
\int_s^T\!\!\!\int_{\Omega} \left\{  -N^{\epsilon_\mu}_{ik}(X,Y)\phi^i_t(X)+ a^{\alpha\beta}_{ij} D_\beta N^{\epsilon_\mu}_{jk}(X,Y) D_\alpha \phi^i(X)\right\} \,dX
\end{multline*}
By writing $\vec \phi= \eta \vec \phi + (1-\eta)\vec \phi$, where $\eta \in \sC^\infty_c(Q_{\bar d_Y}(Y))$ satisfying $\eta =1$ on $Q_{\bar d_Y/2}(Y)$, and using \eqref{eq5.27a}, \eqref{eq5.28b}, and taking $\mu\to \infty$ in the above, we get the identity \eqref{eq2.04x}; see \cite[p. 1662]{CDK} for the details.
To verify the property c), let us assume that $\vec f$ is supported in $\bar \Omega\times (a,b)$, where $a<s<b$ and $\tilde{\vec u}$ be the unique weak solution in $\sV^{1,0}_2(\Omega\times (a,b))^N$ of the problem \eqref{mn.eq2}.
By setting $\tilde{\vec u}(x,t)=0$ for $t>b$ and letting $a\to -\infty$, we may assume that $\tilde{\vec u}$ is defined on the entire $\cQ$.
Then, similar to \eqref{mn.eq3}, we have
\[
\int_\Omega \Phi_{y,\epsilon_\mu}(x) \tilde{u}{}^k(x,s)\,dx =\iint_{\cQ} N^{\epsilon_\mu}_{ik}(X,Y) f^i(X)\,dX.
\]
By the condition A2, it follows that $\tilde{\vec u}$ is locally H\"older continuous in $\cQ$; see the remark we made in deriving \eqref{mn.eq4d}.
By writing $\vec f=\zeta \vec f+ (1-\zeta) \vec f$ and using \eqref{eq5.27a}, \eqref{eq5.28b}, and taking the limit $\mu\to \infty$, we then get 
\[
\tilde{\vec u}(Y)=\int_\cQ \vec N(X,Y)^T \vec f(X)\,dX.
\]
Therefore, we have $\tilde{\vec u}\equiv \vec u$ and thus the property c) is verified.

It is clear from the construction that $\vec N(x,t,y,s)\equiv 0$ if $t<s$.
For any $X=(x,t)\in \cQ$, we define the {mollified Neumann Green's function} of $\sL^{*}$ with a pole at $X$ by letting its $l$-th column to be the unique weak solution in $\sV^{1,0}_2(\Omega\times (T,t))^N$, where $T<t$ is fixed but arbitrary, of the backward problem
\[
\left\{
\begin{aligned}
\sL^{*} \vec u= 0\quad &\text{in }\; \Omega\times(T,t)\\
\partial \vec u/\partial \nu^{*}=  0 \quad&\text{on }\; \partial\Omega\times (T,t)\\
\vec u= \Phi_{x,\epsilon}\vec e_l \quad&\text{on }\; \Omega\times \set{t},
\end{aligned}
\right.
\]
As before, we extend it to the entire $\cQ$ by letting it to vanish on $\Omega\times (t,\infty)$.
By a similar argument as above, we obtain the Neumann Green's function $\vec N^{*}(\cdot, X)$ of $\sL^{*}$ that satisfies the natural counterparts of  the properties of the Neumann Green's function for $\sL$.
Note that the condition A2 together with the estimates i), ii) listed in Remark~\ref{rmk3.8} and its counterparts imply that $\vec N(\cdot,Y)$ and $\vec N^{*}(\cdot,X)$ are locally H\"older continuous in $\cQ\setminus \set{Y}$ and $\cQ\setminus\set{X}$, respectively.
Using the continuity discussed above and proceeding similar to \cite[Lemma~3.5]{CDK}, we find that
\begin{equation}		\label{ctn.eq1}
\vec N(Y,X)=\vec N^{*}(X,Y)^T, \quad \forall X,Y\in \cQ, \quad X\neq Y.
\end{equation}
Moreover, similar to \cite[Eq.~(3.44) and (3.45)]{CDK}, for $t>s$, we have
\[
\vec{N}^\epsilon(x,t,y,s)=\int_\Omega \Phi_{y,\epsilon}(z) \vec N(x,t,z,s)\,dz,
\]
which justifies why we call it ``mollified'', and
\begin{equation}		\label{eq5.30x}
\lim_{\epsilon\to 0+} \vec N^\epsilon(x,t,y,s)=\vec N(x,t,y,s).
\end{equation}
By \eqref{ctn.eq1} and the counterpart of the property c) in Sec.~\ref{sec:ng}, we see that $\vec u$ defined by the formula \eqref{eq3.04c} is a weak solution in $\sV^{1,0}_2(\cQ)^N$ of \eqref{eq3.05a}.

We now prove the identity \eqref{EP.1b} for the weak solution in $\sV^{1,0}_2(\Omega\times(s,\infty))^N$ of the problem \eqref{EP.1b1}.
Similar to \eqref{mn.eq3}, we have
\[
\int_\Omega \Phi_{x,\epsilon}(y) u^k(y,t)\,dy =\int_{\Omega} \hat{N}^\epsilon_{ik}(y,s,x,t) \psi^i(y)\,dy,
\]
where $\hat{\vec N}{}^\epsilon(\cdot,X)$ is the mollified Neumann Green's function of $\sL^{*}$.
The condition A2 implies that $\vec u$ is continuous in $\Omega\times (s,\infty)$, and thus we have
\[
\lim_{\epsilon \to 0} \int_\Omega \Phi_{x,\epsilon}(y) u^k(y,t)\,dy = u^k(x,t).
\]
On the other hand,  by \eqref{eq5.30x} and the counterparts of \eqref{wse.eq2b}, together with the dominated convergence theorem, we get
\[
\lim_{\epsilon \to 0} \int_{\Omega} \hat{N}^\epsilon_{ik}(y,s,x,t) \psi^i(y)\,dy=\int_{\Omega} N^{*}_{ik}(y,s,x,t) \psi^i(y)\,dy.
\]
Then, the identity \eqref{EP.1b} follows from \eqref{ctn.eq1}.
Finally, we obtain \eqref{eq3.07c} similar to \eqref{mn.eq4h} and get \eqref{eq3.08d} from \eqref{eq3.07c} and the condition A2.
\hfill\qedsymbol

\subsection{Proof of Theorem  \ref{EP.thm1}: Case when $\abs{\Omega}=\infty$.}
The proof for the case when $\abs{\Omega}=\infty$ is almost identical to the case when $\abs{\Omega}<\infty$.
As a matter of fact, it is even more simple since we do not need to introduce $\tilde{\vec N}{}^\epsilon(X,Y)$.
From the condition A1, we obtain
\[
\norm{\vec N^\epsilon(\cdot,Y)}_{\sL_{2(n+2)/n}(\Omega\times(a,b))} \le \gamma  \tri{\vec N^\epsilon(\cdot,Y)}_{\Omega\times(a,b)}.
\]
By using the above inequality instead of \eqref{eq5.13d} and follow the same augment as in Section~\ref{sec5.1}, we get the same conclusion.
\hfill\qedsymbol

\subsection{Proof of Theorem  \ref{GE.thm1}}
We adopt the argument in \cite{HK04}, which is in turn based on \cite{Davies89, FS}.
Let $\psi$ be a bounded Lipschitz function on $\bR^n$ satisfying $\abs{D\psi}\le M$ a.e. for some $M>0$ be chosen later.
For $t>s$, we define an operator $P^\psi_{s\to t}$ on $L^2(\Omega)^N$ as follows:
For $\vec f \in L^2(\Omega)^N$ and $T>t$, let $\vec u$ be the unique weak solution in $\sV^{1,0}_2(\Omega\times(s,T))^N$ of the problem
\begin{equation}	\label{eq5.31t}		
\left\{
\begin{array}{ll}
\sL\vec u= 0 \quad &\text{in }\; \Omega\times (s,T),\\
\partial \vec u/\partial \nu=0  &\text{on }\;\partial\Omega \times (s,T),\\
\vec u=e^{-\psi}\vec f &\text{on }\; \Omega\times \set{s}
\end{array}
\right.
\end{equation}
and define $P^\psi_{s\to t} \vec f(x)\coloneq  e^{\psi(x)}\vec u(x,t)$.
Then, by \eqref{EP.1b}, we find
\begin{equation}  \label{eq3.60.3}
P^\psi_{s\to t}\vec f(x)=
e^{\psi(x)}\int_{\Omega}\vec N(x,t,y,s)e^{-\psi(y)}\vec f(y)\,dy.
\end{equation}
Then, as in \cite[Sec.~5.1]{CDK}, we derive
\begin{equation} \label{eq3.70}
\norm{P^\psi_{s\to t}\vec f}_{L^2(\Omega)} \leq e^{\vartheta M^2(t-s)}\norm{\vec f}_{L^2(\Omega)},
\end{equation}
where $\vartheta=\lambda^{-3}$.
We set $\rho= \sqrt{t-s} \wedge R_{M}$ and use the condition A3 to estimate
\begin{align*}
e^{-2\psi(x)}\abs{P^\psi_{s\to t}\vec f(x)}^2  &= \abs{\vec u(x,t)}^2\\
&\leq B_1^2 \rho^{-(n+2)} \int_{t-\rho^2}^t \int_{\Omega \cap B_{\rho}(x)}\abs{\vec u(y,\tau)}^2\,dy\,d\tau\\
& \le B_1^2 \rho^{-(n+2)} \int_{t-\rho^2}^t \int_{\Omega \cap B_{\rho}(x)}e^{-2\psi(y)} \abs{P^\psi_{s\to\tau}\vec f(y)}^2\,dy\,d\tau.
\end{align*}
Thus, by using \eqref{eq3.70}, we derive
\begin{align*}
\abs{P^\psi_{s\to t}\vec f(x)}^2 &\le B_1^2 \rho^{-n-2} \int_{t-\rho^2}^t \int_{\Omega \cap B_{\rho}(x)}e^{2\psi(x)-2\psi(y)} \abs{P^\psi_{s\to\tau}\vec f(y)}^2\,dy\,d\tau\\
& \le B_1^2 \rho^{-n-2} \int_{t-\rho^2}^t \int_{\Omega \cap B_{\rho}(x)}e^{2 \rho M }
\abs{P^\psi_{s\to\tau}\vec f(y)}^2\,dy\,d\tau\\
& \le B_1^2 \rho^{-n-2} \, e^{2 \rho M} \int_{t-\rho^2}^t e^{2\vartheta M^2(\tau-s)}\norm{\vec f}_{L^2(\Omega)}^2\,d\tau\\
& \le B_1^2 \rho^{-n}\,e^{2 \rho M+2\vartheta M^2(t-s)}\norm{\vec f}_{L^2(\Omega)}^2.
\end{align*}
We have thus obtained the following $L^2\to L^\infty$ estimate for $P^\psi_{s\to t}$:
\begin{equation} \label{eq3.70.3}
\norm{P^\psi_{s\to t}\vec f}_{L^\infty( \Omega)} \le  B_1 \rho^{-n/2}\,e^{\rho M+\vartheta M^2(t-s)} \norm{\vec f}_{L^2(\Omega)}.
\end{equation}
We also define the operator $Q^\psi_{t\to s}$ on $L^2(\Omega)^N$ for $s<t$ by setting $Q^\psi_{t\to s}\vec g(y)= e^{-\psi(y)}\vec v(y,s)$, where $\vec v$ is the  weak solution in $\sV^{1,0}_2(\Omega\times(T,t))^N$, where $T<s$, of the backward problem
\begin{equation} \label{eq3.6.11}
\left\{
\begin{array}{ll}
\sL^{*}\vec u= 0 \quad &\text{in }\; \Omega\times (T,t),\\
\partial \vec u/\partial \nu^{*}=0  &\text{on }\;\partial\Omega \times (T,t),\\
\vec u=e^{-\psi}\vec g &\text{on }\; \Omega\times \set{t}.
\end{array}
\right.
\end{equation}
Similar to \eqref{eq3.70.3}, we have
\begin{equation} \label{eq3.73.3}
\norm{Q^\psi_{t\to s}\vec g}_{L^\infty(\Omega)} \le  B_1 \rho^{-n/2}\,e^{\rho M+\vartheta M^2(t-s)} \norm{\vec g}_{L^2(\Omega)}.
\end{equation}
Note that by \eqref{eq5.31t} and \eqref{eq3.6.11} we have
\[
\int_{\Omega}\bigl(P^\psi_{s\to t}\vec f\bigr) \cdot \vec g= \int_{\Omega}\vec f\cdot \bigl(Q^\psi_{t\to s} \vec g\bigr).
\]
Therefore, by duality, \eqref{eq3.73.3} implies that for all $\vec f  \in C^\infty_c(\Omega)^N$,  we have
\begin{equation} \label{eq3.74.3}
\norm{P^\psi_{s\to t}\vec f}_{L^2(\Omega)} \le  B_1 \rho^{-n/2}\,e^{\rho M+\vartheta M^2(t-s)} \norm{\vec f}_{L^1(\Omega)}.
\end{equation}
Now,  set $r=(s+t)/2$ and observe that by the uniqueness, we have
\[
P^\psi_{s\to t}\vec f= P^\psi_{r\to t}(P^\psi_{s\to r}\vec f).
\]
Then, by noting that $t-r=r-s=(t-s)/2$ and $\rho/\sqrt{2}\le \sqrt{t-r}\wedge R_{M} \le \rho$, we obtain from \eqref{eq3.70.3} and \eqref{eq3.74.3} that for  all $\vec f\in C^\infty_c(\Omega)^N$, we have
\[
\norm{P^\psi_{s\to t}\vec f}_{L^\infty(\Omega)} \leq C \rho^{-n}\,e^{ 2 \rho M+\vartheta M^2(t-s)} \norm{\vec f}_{L^1(\Omega)},
\]
where $C=2^{n/2}B_1^2$.
For all $x, y\in \Omega$ with $x\neq y$, the above estimate combined with \eqref{eq3.60.3} yields, by duality, that
\begin{equation} \label{eq3.83}
e^{\psi(x)-\psi(y)}\abs{\vec N(x,t,y,s)} \le  C \rho^{-n}\,e^{ 2 \rho M+\vartheta M^2(t-s)}.
\end{equation}
Let $\psi(z)\coloneq M \psi_0(\abs{z-y})$, where $\psi_0$ is defined on $[0,\infty)$ by
\[
\psi_0(r)=\begin{cases}r&\text{if $r \le \abs{x-y}$} \\
\abs{x-y}&\text{if  $r>\abs{x-y}$}.
\end{cases}
\]
Then, $\psi$ is a bounded Lipschitz function satisfying $\abs{D\psi} \le M$ a.e.
Take $M=\abs{x-y}/2\vartheta(t-s)$ and set $r=\abs{x-y}/\sqrt{t-s}$.
By \eqref{eq3.83} and the obvious inequality $\rho/\sqrt{t-s}\le 1$, we have
\[
\abs{\vec N(x,t,y,s)}\le  C\rho^{-n}\, \exp(r/\vartheta-r^2/4\vartheta).
\]
Let $A=A(\vartheta)=A(\lambda)$ be chosen so that
\[
\exp(\tau/\vartheta-\tau^2/4\vartheta)\le A\exp(-\tau^2/8\vartheta),\quad\forall \tau\in [0,\infty).
\]
If we set $\kappa=1/8\vartheta=\lambda^3/8$, then we obtain
\[
\abs{\vec N(x,t,y,s)}\le C\rho^{-n}\exp\left\{-\kappa\abs{x-y}^2/(t-s)\right\}.
\]
where $C=C(n,N,\lambda,B_1)$.
We have thus proved \eqref{GE.thm1.1}.
Finally, \eqref{GE.thm1.1} gives
\[
\abs{\vec N(x,t,y,s)}\le C\left\{1\vee ((t-s)/R_{M}) \right\}^{n/2}\, (t-s)^{-n/2}\exp\left\{-\kappa \abs{x-y}^2/(t-s)\right\},
\]
which clearly implies \eqref{eq2.17dx} if $0<t-s\le T$.
\hfill\qedsymbol

\subsection{Proof of Theorem  \ref{GE.thm2}}
Notice that \eqref{GE.thm2.1} implies that
\begin{equation}					\label{eq5.18}
\abs{\vec N(X,Y)} \leq C \abs{X-Y}_\sP^{-n}\quad \text{if }\, 0<\abs{t-s} <R_{M}^2,
\end{equation}
where $C=C(n,N,\lambda, B_2, \kappa)$.
Then, by the energy inequality (see \cite[Eq. (3.21)]{CDK}) and noting that $R_{M}/2$ and $R_{M}$ are comparable to each other when $R_{M}<\infty$, we derive from \eqref{eq5.18} that for $0<r<R_M$, we have
\begin{equation} \label{eq5.15}
\tri{\vec N(\cdot,Y)}_{\cQ \setminus \bar Q_r(Y)} \le Cr^{-n/2}.
\end{equation}
Let $X\in\cQ$ and $0<R<R_{M}$ be given.
Without loss of generality, we assume $X=0$ and write $Q_R=Q_R(0)$, etc.
Let $\vec u$ be a weak solution of
\[
\sL^{*} \vec u = 0 \text{ in } Q=Q_R^{+}\cap \cQ,\quad \partial \vec u/\partial \nu^{*} =0 \text{ on } S=Q_R^{+}\cap \partial_p \cQ.
\]
We will show that $\vec u$ is locally bounded in $\frac{1}{2}Q$, where we use notation $\alpha Q=Q^{+}_{\alpha R}\cap \cQ$, and satisfies the estimate \eqref{LB}.
The other case can be treated in the same fashion.
Let $\vec w=\eta \vec u$, where $\eta \in \sC^\infty_c(Q_{R/2})$ is such that
\[
0\leq \eta \le 1,\quad \eta \equiv 1\text{ on }Q_{3R/8},\quad \abs{D_x\eta}^2+\abs{\eta_t} \le 32/R^2 .
\]
Then, $\vec w$ becomes a weak solution of
\[
\left\{
\begin{array}{ll}
\sL^{*} \vec w= - \eta_t \vec u - (A^{\beta\alpha})^T D_\beta \vec u D_\alpha \eta - D_\alpha \left( (A^{\alpha\beta})^T D_\beta \eta \vec u\right)
 &\text{in } \Omega\times (0,\infty),\\
\partial \vec w/\partial \nu^{*}= (A^{\beta\alpha})^TD_\beta \eta \vec u \nu_\alpha &\text{on } \partial \Omega \times (0,\infty).
\end{array}
\right.
\]
For $Y=(y,s) \in \frac{1}{4}Q$, let $\vec N^\epsilon(\cdot,Y)$ be the mollified Neumann Green's function as constructed in the proof of Theorem~\ref{EP.thm1}.
Then similar to \eqref{mn.eq3}, we have
\begin{multline*}
\int_\Omega \Phi_{y,\epsilon}(x) \eta(x,s) u^k(x,s)\,dx 
= -\int_{\frac{1}{2} Q} \eta_t u^i N^\epsilon_{ik}(\cdot,Y)\\
-\int_{\frac{1}{2} Q} a^{\beta\alpha}_{ji}  D_\beta u^j D_\alpha\eta N_{ik}^\epsilon(\cdot,Y)+ \int_{\frac{1}{2} Q} a^{\beta\alpha}_{ji} D_\beta \eta u^j D_\alpha N_{ik}^\epsilon(\cdot,Y).
\end{multline*}
Since $\abs{X-Y}_\sP> R/8$ for $X \in Q\cap \supp (\abs{D_x \eta}+\abs{\eta_t})$, the integrals on the right hand side are in fact integrals over $\frac{1}{2}Q \setminus Q_{R/8}(Y)$.
Then, we take limit $\epsilon$ to zero in the above by using \eqref{eq5.27a}, \eqref{eq5.28b} to get
\begin{multline}				\label{eq4.28v}
u^k(Y) = - \int_{\frac{1}{2}Q}  \eta_t  u^i  N_{ik}(\cdot,Y)- \int_{\frac{1}{2}Q} a^{\alpha\beta}_{ij}  N_{jk}(\cdot,Y)D_\alpha u^i D_\beta\eta\\
+\int_{\frac{1}{2}Q} a^{\alpha\beta}_{ij} D_\beta  N_{jk}(\cdot,Y)u^i D_\alpha \eta \eqcolon  I_1+I_2+I_3.
\end{multline}
Denote $A_R(Y)\coloneq \cQ\cap \left(Q_{3R/4}(Y)\setminus Q_{R/8}(Y)\right) \supset \frac{1}{2}Q\setminus Q_{R/8}(Y)$.
By the properties of $\eta$, H\"older's inequality, and \eqref{eq5.15}, we estimate
\begin{align*}
\abs{I_1}  & \le C R^{-2} \norm{\vec N(\cdot,Y)}_{\sL_{2,\infty}(A_R(Y))}\, \norm{\vec u}_{\sL_{2,1}(\frac{1}{2}Q)} \le C R^{-(n+2)/2} \norm{\vec u}_{\sL_2(\frac{1}{2}Q)}.\\
\abs{I_3}  & \le C R^{-1} \norm{D_x\vec N(\cdot,Y)}_{\sL_2(A_R(Y))}\, \norm{\vec u}_{\sL_2(\frac{1}{2}Q)} \le C R^{-(n+2)/2} \norm{\vec u}_{\sL_2(\frac{1}{2}Q)}.
\end{align*}
Also, by using \eqref{eq5.18} and the energy inequality, we estimate
\[
\abs{I_2}  \le C R^{-1} \norm{\vec N(\cdot,Y)}_{\sL_2(A_R(Y))}\, \norm{D \vec u}_{\sL_2(\frac{1}{2}Q)} \le C R^{-(n+2)/2} \norm{\vec u}_{\sL_2(Q)}.
\]
By combining above estimates for $I_1, I_2$, and $I_3$, we conclude from \eqref{eq4.28v} that
\[
\norm{\vec u}_{\sL_\infty(\frac{1}{4}Q)} \le C R^{-(n+2)/2} \norm{\vec u}_{\sL_2(Q)},
\]
where $C=C(n,N,\lambda, B_2,\kappa)$.
Then, we obtain the condition A3 from the above inequality by a standard covering argument.
\hfill\qedsymbol

\subsection{Proof of Theorem~\ref{thm4.4a}}
The idea of proof is the same as that of \cite[Theorem~2.12]{DK09} but we reproduce some key steps of the proof for the completeness.
Let $\vec K(x,y,t)=\tilde{\vec N}(x,t,y,0)$, where $\tilde{\vec N}(x,t,y,s)$ is as in \eqref{ntilde} and $\vec N(x,t,y,s)$ is the Neumann Green's function of $\sL=\partial/\partial t +L$ in $\cQ$.
We define the Neumann function $\vec G(x,y)$ by
\begin{equation}
\label{eq:g01}
\vec G(x,y)\coloneq  \int_0^\infty \vec K(x,y,t)\,dt.
\end{equation}
Since $\Omega$ is an extension domain and $\abs{\Omega}<\infty$, there is a constant $\varrho=\varrho(\Omega)$ such that for all $u \in \tilde W^{1,2}(\Omega)$, we have
\begin{equation}				\label{eq002v}
\norm{u}_{L^2(\Omega)}\le \varrho \norm{D u}_{L^2(\Omega)}.
\end{equation}
Then  by \cite[Lemma~3.12]{DK09}, we find that the integral in \eqref{eq:g01} is absolutely convergent for $x \neq y$ and thus $\vec G(x,y)$ is well defined.
Also, it is clear that $\int_\Omega \vec G(x,y)\,dx =0$.
We define the Neumann function $\vec G^{*}(x,y)$ of $L^{*}$ similarly.
Then, by \eqref{ctn.eq1} we find that (see \cite[Eq.~(3.21)]{DK09})
\[
\vec G^{*}(x,y)= \int_0^\infty \tilde{\vec N}{}^{*}(x,-t,y,0)\,dt=\int_0^\infty \vec K(y,x,t)^T\,dt=\vec G(y,x)^T,
\]
where $\tilde{\vec N}{}^{*}(x,t,y,s)=\vec N^{*}(x,t,y,s)-\frac{1}{\abs{\Omega}}H(t-s)\vec I$ and $\vec N^{*}(x,t,y,s)$ is the Neumann Green's function for $\sL^{*}=-\partial/\partial t +L^*$.
Let us denote
\[\bar{\vec K}(x,y,t)=\int_0^t \vec K(x,y,s)\,ds.\]
\begin{lemma}			\label{lem4.10a}
The following estimates hold uniformly for $t>0$ and $y\in\Omega$.
\begin{enumerate}[a)]
\item
$\norm{\bar{\vec K}(\cdot,y,t)}_{L^p(B(y,d_y))} \le C_p (d_y^2+\varrho^2) d_y^{2/p-2},\quad \forall p \in [1,2)$.
\item
$\norm{D_x\bar{\vec K}(\cdot,y,t)}_{L^p(B(y,d_y))}\le C_p (d_y^2+\varrho^2) d_y^{-3+2/p},\quad \forall p \in [1,4/3)$.
\item
$\norm{\bar{\vec K}(\cdot,y,t)}_{L^{4}(\Omega\setminus B(y,r))}
\le C(r^2+\varrho^2) r^{-3/2},\quad \forall r \in (0, d_y]$.
\item
$\norm{D_x\bar{\vec K}(\cdot,y,t)}_{L^2(\Omega\setminus B(y,r))} \le C (r^2+\varrho^2) r^{-2},\quad \forall r \in (0, d_y]$.
\end{enumerate}
\end{lemma}
\begin{proof}
See \cite[Lemma~3.23]{DK09}.
\end{proof}

We have to show that $\vec G(x,y)$ defined by the formula \eqref{eq:g01} satisfies the properties i) -- iii) in Section~\ref{sec:4.2enf}.
We focus on iii), which may seem less clear than i) or ii).
For any $\vec f\in C^\infty_c(\bar \Omega)$ satisfying $\int_\Omega \vec f(x)\,dx =0$, let $\vec u$ be defined by \eqref{choej}. 
Note that
\[
\vec v(x,t) \coloneq \int_\Omega \bar{\vec K}(y,x,t)^T \vec f(y)\,dy= \int_0^t \!\!\!\int_\Omega \vec K(y,x,s)^T \vec f(y)\,dy\, ds
\]
is absolutely convergent by Lemma~\ref{lem4.10a}, and
\begin{align}
\label{eq:E29a}
\lim_{t\to\infty} \vec v(x,t)&=\int_\Omega \vec G(y,x)^T \vec f(y)\,dy=\vec u(x),\\
\nonumber
\vec v_t(x,t)&=\int_\Omega \vec K(y,x,t)^T \vec f(y)=\int_\Omega \vec N^{*}(x,-t,y,0) \vec f(y)\,dy,
\end{align}
where we have used the assumption $\int_\Omega \vec f=0$ in the last step.
Therefore, for any $T>0$, $\vec w \coloneq \vec v_t$ is the weak solution in $\sV^{1,0}_2(\Omega\times(0,T))^N$ of the problem
\[
\left\{
\begin{aligned}
\vec w_t+L^{*} \vec w= 0 \quad &\text{in }\; \Omega\times (0,T)\\
\partial \vec w/\partial \nu^{*}= 0 \quad&\text{on }\; \partial\Omega\times(0,T)\\
\vec w(\cdot, 0)=\vec f \quad&\text{on }\; \Omega,
\end{aligned}
\right.
\]
Then, as in the proof of \cite[Lemma~3.12]{DK09}, for $t>0$, we have
\begin{equation}		\label{eq:E29b}
\norm{\vec v_t(\cdot,t)}_{L^2(\Omega)} \le C e^{-\lambda\varrho^{-2}t}\norm{\vec f}_{L^2(\Omega)},
\end{equation}
which implies $\vec v_t(\cdot,t)\in L^2(\Omega)^N$.
Also, by using $\vec N^{*}(x,t,y,s)=\vec N^{*}(x,t-s,y,0)$, we find
\[
\vec v(x,t)=\int_{-t}^0 \int_\Omega \vec N^{*}(x,-t,y,s) \vec f(y)\,dy\, ds,
\]
and thus, for $T>0$, $\vec v$ is the weak solution in $\sV^{1,0}_2((0,T)\times\Omega)^N$ of the problem
\begin{equation}		\label{eq5.49jw}
\left\{
\begin{aligned}
\vec v_t+L^{*}\vec v= \vec f\quad &\text{in }\; \Omega\times (0,T)\\
\partial \vec v/\partial \nu^{*}= 0 \quad&\text{on }\; \partial\Omega\times(0,T)\\
\vec v(\cdot, 0)=0 \quad&\text{on }\; \Omega.
\end{aligned}
\right.
\end{equation}
Then, by using \eqref{eq:E29b} and \eqref{eq5.49jw}, we obtain (see \cite[Eq.~(3.47)]{DK09})
\[
\norm{D_x \vec v(\cdot,t)}_{L^2(\Omega)} \le C \norm{\vec f}_{L^2(\Omega)}, \quad\text{a.e. }t>0.
\]
Therefore, by  \eqref{eq002v} for a.e. $t>0$, we have
\[
\norm{\vec v(\cdot,t )}_{W^{1,2}(\Omega)} \le C \norm{\vec f}_{L^2(\Omega)}.
\]
Then, by the weak compactness of the space $W^{1,2}(\Omega)$, we find that there is a sequence $t_m\to \infty$ such that $\vec v(\cdot, t_m) \rightharpoonup \tilde{\vec u}$ weakly in $W^{1,2}(\Omega)^N$ for some $\tilde{\vec u} \in W^{1,2}(\Omega)^N$.
By \eqref{eq:E29a}, we must have $\vec u=\tilde{\vec u}$ and similar to \cite[Eq.~(3.48)]{DK09} and the argument after it, we find $\vec u$ is a weak solution of the problem \eqref{hee}.
By using \eqref{eq:E23}, it is easy to see that $\int_\Omega \vec u =0$.
We have thus verified the property iii).
By repeating essentially the same proof of \cite[Theorem~2.12]{CDK}, we find that $\vec G(x,y)$ satisfies the properties i) and ii) as well, and that it has the logarithmic bound \eqref{eq4.03mr}.

If $\Omega$ is a Lipschitz domain, then by modifying \cite[Lemma~4.4]{DK09}, one can show that $\sL=\partial/\partial t+L$ and $\sL^{*}=-\partial/\partial t+L^{*}$ satisfy the condition A3 with $R_M=\diam \Omega=d$ and $B_1$ depending on $\lambda, N$ and Lipschitz character of $\partial\Omega$; see Section~\ref{sec:4.1.3}.
By \eqref{eq5.18}, for $X=(x,t)\in \cQ$ satisfying $\abs{t}<d^2$, we have
\begin{equation}					\label{eq618z}
\abs{\vec K(x,y,t)} \le C\abs{X-\bar{Y}}_{\sP}^{-2},\quad \bar Y=(y,0),
\end{equation}
and by \eqref{eq5.15}, for $0<r < d$, we have
\[
\tri{\vec K(\cdot,y)}_{\cQ \setminus Q(\bar{Y},r)} \le Cr^{-1}.
\]
Similar to \cite[Eq.~(3.59)]{DK09}, for $0<r < d$ and $t \ge 2r^2$, we have
\begin{equation}							\label{eq12.27g}
\abs{\vec K(x,y,t)} \le C r^{-2}e^{-\lambda \varrho^{-2} (t-2r^2)}.
\end{equation}
We set $r \coloneq \frac{1}{2}\min(\varrho, d)$.
If $0<\abs{x-y} \le  r$, then by \eqref{eq:g01}, we have
\[
\abs{\vec G(x,y)} \le \int_0^{\abs{x-y}^2} + \int_{\abs{x-y}^2}^{2r^2}+\int_{2 r^2}^\infty \abs{\vec K(x,y,t)}\,dt \eqcolon  I_1+I_2+I_3.
\]
It then follows from \eqref{eq618z} and \eqref{eq12.27g} that
\begin{align*}
I_1 &\le C\int_0^{\abs{x-y}^2}\abs{x-y}^{-2}\,dt \le C,\\
I_2&\le C\int_{\abs{x-y}^2}^{2r^2}t^{-1}\,dt\le  C+C\ln (r/\abs{x-y}),\\
I_3&\le C\int_{2r^2}^\infty r^{-2} e^{-\lambda \varrho^{-2} (t-2r^2)}\,dt\le C \varrho^2 r^{-2}.
\end{align*}
Combining all together we conclude that if $0<\abs{x-y}\le r=\frac{1}{2}\min(\varrho, d)$, then
\begin{equation}			\label{eq15.08b}
\abs{\vec G(x,y)} \le C\left(1+(\varrho/r)^2+ \ln (r/d)+ \ln (d/\abs{x-y}) \right).\\
\end{equation}
If $\abs{x-y} \ge r=\frac{1}{2}\min(\varrho,d)$, then by \eqref{eq618z} and \eqref{eq12.27g} we have
\begin{multline}							\label{eq6.23ww}
\abs{\vec G(x,y)} \le \int_0^{2r^2}+\int_{2r^2}^\infty \abs{\vec K(x,y,t)}\,dt\\
\le C \int_0^{2r^2} r^{-2} + C \int_{2r^2}^\infty \ r^{-2} e^{-\lambda\varrho^{-2}(t-2r^2)}\,dt \le C+ C \varrho^2 r^{-2}.
\end{multline}
By combining \eqref{eq15.08b} and \eqref{eq6.23ww}, and using $d <\infty$, we get \eqref{eq13.03k}.
\hfill\qedsymbol

\begin{acknowledgment}
We thank Steve Hofmann, Fritz Gesztesy, and Marius Mitrea for helpful discussion and correspondence.
Seick Kim is supported by TJ Park Junior Faculty Fellowship.
\end{acknowledgment}



\end{document}